\documentclass[11pt]{article}

\usepackage{amsmath}
\usepackage{epsfig, amssymb, amsfonts}
\usepackage{amsthm}
\usepackage{amstext}
\usepackage{amsopn}
\usepackage{mathrsfs}
\usepackage{subfigure}
\usepackage{graphicx}
\usepackage[left=2.3cm,top=2cm,right=2.3cm,bottom=2cm, nohead,foot=1cm]{geometry}
\usepackage[makeroom]{cancel}
\usepackage{color}
\usepackage{enumerate}
\usepackage{comment}

\numberwithin{equation}{section}

\newtheorem{theorem}{Theorem}[section]

\newtheorem{lemma}[theorem]{Lemma}
\newtheorem{conjecture}[theorem]{Conjecture}
\newtheorem{corollary}[theorem]{Corollary}
\newtheorem{proposition}[theorem]{Proposition}
\newtheorem*{theorem*}{Theorem}
\newtheorem*{proposition*}{Proposition}

\theoremstyle{definition}

\newtheorem{remark}[theorem]{Remark}

\newcommand{\R}{{\mathbb R}}

\newcommand{\Z}{{\mathbb Z}}

\newcommand{\sbf}{{s}}
\newcommand{\bbf}{{b}}

\newcommand{\Try}{{\triangleleft}}

\date{\vspace{-9ex}}

\title{High-temperature scaling limit for directed polymers on a  hierarchical   lattice with bond disorder \vspace{.4cm}}

\date{  }

  \author{ \textbf{Jeremy Thane Clark}\footnote{ {\tt
jeremy@olemiss.edu}}\vspace{.1cm}  \\  University of Mississippi, Department of Mathematics   }

\providecommand{\keywords}[1]{\textbf{\textit{Keywords:}} #1}

\begin{document}
\maketitle

\begin{abstract} 
 Diamond ``lattices" are sequences of recursively-defined graphs that provide a network of directed pathways between two fixed root nodes, $A$ and $B$.  The construction recipe for diamond graphs  depends on a branching number $b\in \mathbb{N}$ and a segmenting number $s\in \mathbb{N}$, for which a larger value of the ratio $s/b$ intuitively corresponds to more opportunities for intersections between two randomly chosen paths.   By attaching  i.i.d.\ random variables to the bonds of the graphs, I construct a random Gibbs measure on the set of directed paths by assigning each path an ``energy" given by summing the random variables along the path.  For the case $b=s$, I propose a scaling regime in which the temperature grows along with the number of hierarchical layers of the graphs, and the  partition function (the normalization factor of the Gibbs measure) appears to converge in law. I prove that all of the positive integer moments of the partition function converge in this limiting regime.  The motivation of this work is to prove a functional limit theorem that is analogous to a previous result obtained in the $b<s$ case.

\end{abstract}


\section{Introduction}

 The phrase \textit{directed polymers in disordered environments} refers to a class of models for a randomized path whose probabilistic law is influenced by random local impurities scattered throughout the medium.  In this framework, there are two layers of randomness: the microscopic arrangement of the impurities in the environment forms the underlying layer and the path of the polymer through the medium, viewed as a random walk, forms the next layer.  The term \textit{directed} means that the polymer's path is partially restricted such that it maintains progress along a select axis while remaining free to wander  in its other spatial degrees of freedom.  This structure prevents the polymer from forming loops, and thus revisiting impurities.  An interesting but challenging question arises for these models in the limit that the polymer's length grows: are the environmental impurities essentially determinative of  the polymer's course through the medium or do they have a marginal, merely quantitative effect on the the polymer's statistics as a random walk?  These behavioral outcomes are referred to, respectively, as \textit{strong disorder} and \textit{weak disorder}.  The following  factors  modulate the effect of the impurities on the polymer's law:
\begin{enumerate}[(I)]

\item Increasing the temperature of the system decreases the influence of the  impurities.

\item Increasing the number of degrees of freedom available to the polymer weakens the influence of the  impurities.

\end{enumerate}
 (I) is directly equivalent to decreasing the strength of the impurities within the Gibbsian formalism of the model, and (II)  follows indirectly as an averaging effect generated by the greater number of available paths.   Classical results~\cite{Bolth,CometsII} in the field of disordered directed polymers show that when the  environmental impurities lie on the rectangular lattice $\Z_+\times \Z^d$, i.e., the  $(1+d)$-dimensional directed polymer, then for $d=1,2$  strong disorder behavior prevails for all finite temperatures $\beta^{-1}>0$, and for $d\geq 3$ there is a critical point temperature $\beta^{-1}_c\in (0,\infty)$ such that strong and weak disorder hold for temperatures below and above $\beta^{-1}_c$, respectively.  Formally, the critical temperature in the $d=1,2$ cases is $\beta^{-1}_c=\infty$ since weak disorder is present only at infinite temperature, $\beta=0$, in which disorder is completely absent.
  For a review of  results in the field of disordered directed polymers, see the recent book~\cite{CometsBook} by F.\ Comets.

From a probability perspective, one natural approach to these models is to search for white noise scaling limits in which the number of impurities (i.e., the system size) grows along with a counterbalancing effective decay in the strength of the individual impurities (controlled indirectly, for instance,  through a rising temperature).  In such a limiting regime, the random environmental impurities would be homogenized into a white noise field driving the limiting model.  A celebrated result of this type was obtained by Alberts, Khanin, and Quastel~\cite{AKQ} for the $(1+1)$-rectangular lattice  polymer.  Their work arrives at an interesting and conceptually important disordered continuum polymer that is formally related to the 1-d stochastic heat equation~\cite{AKQII}.   Caravenna, Sun, and Zygouras~\cite{CSZ1,CSZ3} recently introduced an analogous distributional limit theorem for  the $(1+2)$-polymer, which employs a  unified technique that is also applicable to  the $(1+1)$-polymer in a range of cases with heavy-tailed disorder variables.  Infinite temperature white noise limits of this type are not suitable  for the $(1+d)$-polymer when $d\geq 3$ since the critical point temperature $\beta_c^{-1}$ is finite.  

Although rectangular lattices are the most mathematically compelling graphical structures for studying  directed polymers in disordered enviornments--due, in part, to their limiting connection with the stochastic heat equation--, it is also interesting to explore analogous models on graphical structures that have contrasting characteristics, such as exact hierarchical symmetry.  The \textit{diamond hierarchical lattice} is one such toy structure that researchers have chosen  to grow their understanding of disordered polymers~\cite{Cook,lacoin,ACK} and a variety of other statistical mechanical phenomena such as pinning models~\cite{GLT, lacoin3}, resister networks~\cite{Spohn,Hambly,Goldstein}, diffusion on fractals~\cite{HamblyII}, and spin models~\cite{Griffiths}.  Diamond hierarchical lattices  are  sequences $\big\{D^{b,s}_n\big\}_{n\in \mathbb{N}}$ of recursively-defined finite graphs whose construction depends on a branching parameter $b\in \{2,3,\cdots\}$ and a segmenting parameter $s\in \{2,3,\cdots\}$ (see next section for  details). Lacoin and Moreno~\cite{lacoin} considered disordered polymers on the diamond lattice with disorder variables placed on the sites and proved that the polymers exhibit strong disorder for $s\geq b$ and weak disorder for $s<b$.  Thus,  the respective cases $s>b$, $s=b$, and $s<b$ are analogous to $d<2$, $d=2$, and $d>2$ of the rectangular lattice, and this correspondence can be understood heuristically  based on  the expected  number of sites shared by two randomly chosen directed paths; see the introduction of~\cite{lacoin}.     In~\cite{ACK} we studied an   infinite-temperature distributional scaling limit analogous to~\cite{AKQ}  for the partition function   of the diamond lattice polymer  in the case when $s>b$ and disorder variables are placed on either sites or bonds.   The techniques of~\cite{ACK} do not extend to proving a  limit theorem  for the partition function in the $s=b$ case, where it is relatively tricky to determine a plausible choice of  infinite-temperature scaling $\beta^{-1}\equiv \big(\beta_{n}^{b}\big)^{-1}\nearrow \infty$ as the generation, $n$, of the diamond graph $D^{b,b}_n$ grows.

 In this article I propose an infinite-temperature scaling for the $s=b$ diamond lattice polymer in the case of  bond disorder.   My main result is to prove that under this proposed limiting regime  the positive integer moments of the partition function  converge as the generation of the diamond graphs tends to infinity.  This extends results in~\cite{AC}, but falls short of proving a functional limit theorem for the partition function because the limiting moments increase at a super-factorial rate.  A similar analysis would likely apply to the analogous model with site disorder, but  messier estimates would be required.

\subsection{The model,  scaling limit, and main results}\label{SecMain}

\noindent \textbf{(a). Construction of the diamond graphs}\vspace{.2cm}\\
Fix a branching number $b\in \mathbb{N}$ and a segmenting number  $\sbf \in \mathbb{N}$.  The sequence of diamond graphs $D_{n}^{b,s}$ are defined inductively as follows:
\begin{itemize}
\item  $D_{0}^{b,s}$ is comprised of two  vertices, $A$ and $B$, and a single bond connecting them.

\item   $D_{n}^{b,s}$ is constructed from $D_{n-1}^{b,s}$  by replacing each bond on $D_{n-1}^{b,s}$ by a subgraph formed by
 $\bbf $ branches that are each split into $\sbf $ segments; see figure below.
\end{itemize}

\begin{center}
\includegraphics[scale=.6]{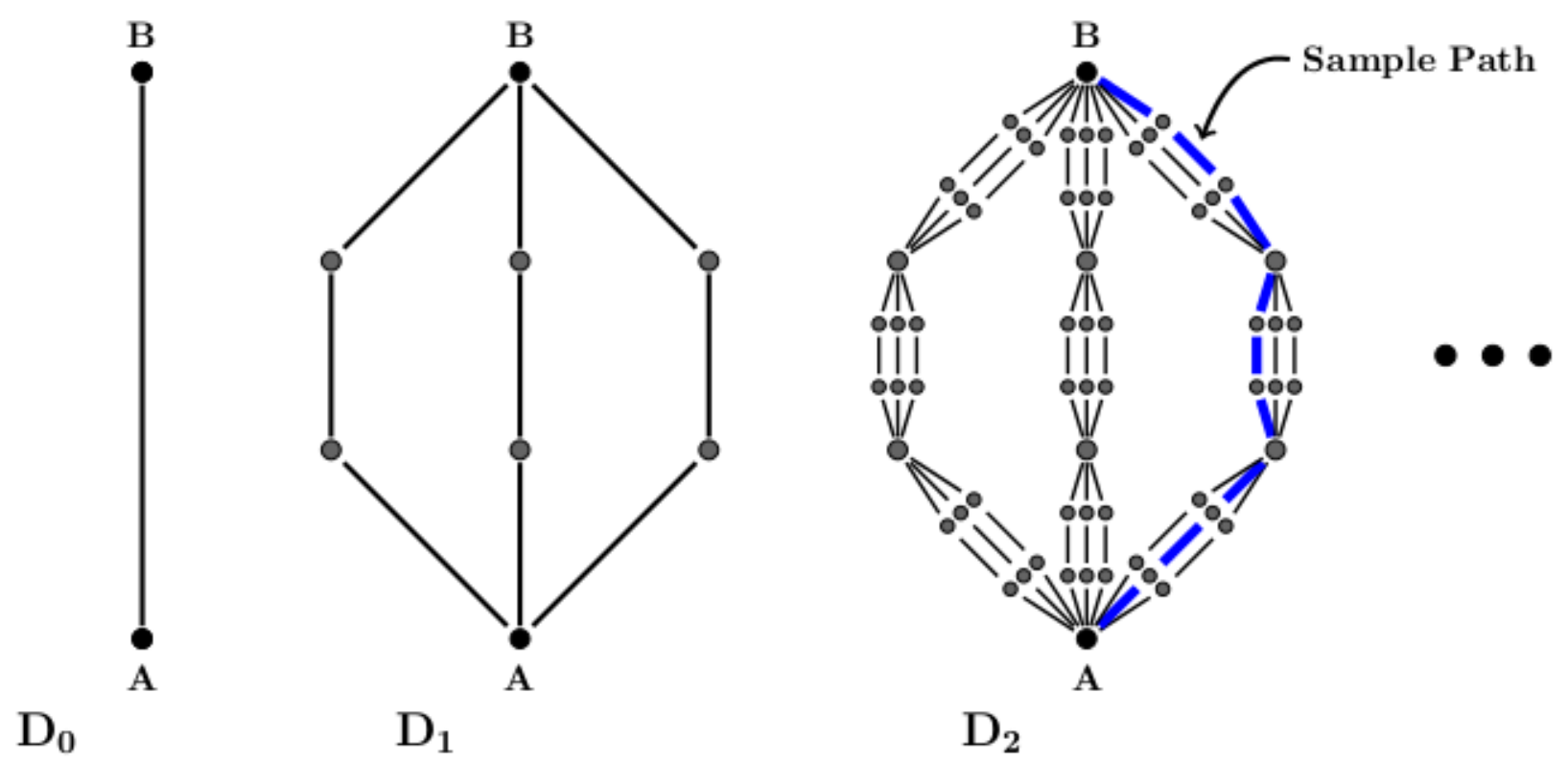}\\
\small The first three recursively-defined diamond graphs with $\bbf=3$  and $\sbf=3$.
\end{center}
A path on $D_{n}^{b,s}$ is said to be \textit{directed} if it begins at $A$ and moves monotonically to $B$ as its termination point.  The set of bonds and directed paths on $D_n^{b,s}$ are denoted by $E_n^{b,s}$ and $\Gamma_{n}^{b,s}$, respectively. \vspace{.4cm}

\noindent \textbf{(b). Randomized Gibbsian measure on paths}\vspace{.2cm}\\
Next let us build a statistical model by attaching  i.i.d.\ random variables $\omega_a$ to each bond $a\in E_n^{b,s}$ of the diamond graph $D_n^{b,s}$ with the following standard assumptions on the variables: mean zero, variance one, and finite exponential moments, $  \mathbb{E}[ e^{\beta \omega_{a}  }  ]$,  for sufficiently small $\beta>0$.   The  ``energy" assigned to a directed path $p\in \Gamma_n^{b,s}$ is the random quantity defined by
$$ H_{n}^{\omega}(p) \, :=  \, \sum_{a\Try p} \omega_{a}\,,$$
where the summation is over all bonds, $a$, lying along the path $p$. Given an inverse temperature value $\beta\in [0,\infty)$,  the Gibbs formalism defines a random probability measure on paths
$$\mu^{(\omega)}_{\beta, n}(p) \,   = \, \frac{  e^{\beta  H_{n}^{\omega}(p) }    }{  Z_{n}^{\omega}(\beta)   } \,  ,   $$
where the \textit{partition function}, $ Z_{n}^{\omega}(\beta)$,  normalizes the  measure: 
$$Z_{n}^{\omega}(\beta) \, := \,\sum_{p\in \Gamma_{n}^{b,s}  }e^{\beta  H_{n}^{\omega}(p) } \, .
     $$

My focus will be on the distributional behavior of the random variables $Z_{n}^{\omega}(\beta)$  in joint limits in which  the temperature $\beta^{-1} \equiv \beta_n^{-1}$ grows along with the number, $n$, of hierarchical layers of the system.  It is convenient to frame this analysis in terms of a normalized version of the partition function:
\begin{align}\label{ExpW}
W_{n}(\beta)\, := \, \frac{  Z_{n}^{\omega}(\beta)  }{  \mathbb{E}\big[  Z_{n}^{\omega}(\beta)   \big] }\, = \, \frac{1}{|\Gamma_{n}^{b,s}|  }\sum_{p\in \Gamma_{n}^{b,s}  }\prod_{a\Try p} \frac{e^{\beta  \omega_{a}  } }{ \mathbb{E}[e^{\beta  \omega_{a}  } ]} \,.
\end{align}

\vspace{.4cm}

\noindent \textbf{(c). Distributional and variance recursion relations}\vspace{.2cm}\\
From the constructive definition of the diamond graphs, it follows that $D_{n+1}^{b,s}$ is built from $b\cdot s$ embedded copies of $D_n^{b,s}$.  Hence if $W_{n}^{(i,j)}(\beta)$ are i.i.d.\ copies of $W_{n+1}(\beta)$, then I have the distributional recurrence  relation
\begin{align}\label{Induct}W_{n+1}(\beta)  \, \stackrel{ d }{ =}  \,  \frac{1}{b}\sum_{i=1}^{b} \prod_{j=1}^{s}  W_{n}^{(i,j)}(\beta) \hspace{1cm}\text{with}\hspace{1cm}   W_{0}(\beta)\,\stackrel{d}{=}\,\frac{e^{\beta\omega}}{\mathbb{E}\big[e^{\beta\omega}\big]}     \, .
\end{align}
Since the variables $W_{n}(\beta)$ have mean one, the above implies that the variances $\varrho_n(\beta):=\textup{Var}\big(W_{n}(\beta)\big)$ are related through 
\begin{align}\label{RecurVar}
\varrho_{n+1}(\beta)=M_{b,s}\big(\varrho_n(\beta)\big) \hspace{1cm} \text{where}\hspace{1cm}   M_{b,s}(x)\,:=\,\frac{1}{b}\Big[(1+x)^s\,-\,1   \Big] \,.  
\end{align}   
By induction,
\begin{align}\label{VarIterative}
\varrho_n(\beta)\,=\, M_{b,s}^n(V_\beta) \hspace{1cm}\text{for}\hspace{1cm} V_{\beta}:= \textup{Var}\big(  W_{0}(\beta)\big)  \,, 
\end{align} 
where $M_{b,s}^n$ refers to the $n$-fold composition of the polynomial maps $M_{b,s}:[0,\infty)\rightarrow [0,\infty)$. 
Notice that $x=0$ is a fixed point for  $M_{b,s}$ and that for $0<x\ll 1$
\begin{align}\label{TheEm}
 M_{b,s}(x)\,=\,  \begin{cases} \frac{s}{b}x +\mathcal{O}(x^2) & \quad  s\neq b\,,  \\  x+\frac{b-1}{2}x^2+\frac{(b-1)(b-2)}{6}x^3+\mathit{O}(x^4)     & \quad  s=b\,.  \end{cases}  
 \end{align}
Thus, the fixed point at zero is linearly repelling for $s>b$, marginally repelling for $s=b$, and linearly attracting for $s<b$.     
  
\vspace{.4cm}

\noindent \textbf{(d). Tuning high-temperature scaling limits through the variance}\vspace{.2cm}\\
To briefly discuss plausible large $n$ scaling limits   for the random variables $W_{n}(\beta)$  with vanishing inverse temperature   $\beta\equiv \beta_{n}^{b,s}\searrow 0$, let us consider  asymptotics for $\beta_{n}^{b,s}$ that are fine-tuned with a  decay rate  such that the variances $\varrho_n\big(\beta_{n}^{b,s}\big)$ converge.    In~\cite{ACK}, it was  shown that when $s>b$, the sequence $W_{n}\big(\beta_{n,r}^{b,s}\big)$ converges in law with large $n$ for any fixed  value of the parameter $r\in \R_+$, where
\begin{align}\label{BetaFormb<s}
\beta_{n,r}^{b,s}=\sqrt{r}\Big(\frac{b}{s}  \Big)^{n/2}\,+\,\mathit{o}\bigg(\Big(\frac{b}{s}  \Big)^{n/2}\bigg)\,.
\end{align}
This is plausible given~(\ref{VarIterative}) and~(\ref{TheEm}) since with $n\gg 1$
\begin{align}\label{XFormb<s}
V_{\beta_{n,r}^{b,s}}\,:=\, \textup{Var}\Bigg( \frac{ e^{\beta_{n,r}^{b,s}\omega}}{\mathbb{E}\big[e^{\beta_{n,r}^{b,s}\omega}\big] }\Bigg) \,=\, r\Big(\frac{b}{s}\Big)^{n}+\mathit{o}\bigg(\Big(\frac{b}{s}\Big)^{n}\bigg)  \,,
\end{align}
and  the variance $\varrho_n\big(\beta_{n,r}^{b,s}\big)$ converges to a limit function $R_{b,s}(r)$ satisfying $ M_{b,s}\big( R_{b,s}(r)  \big)=R_{b,s}(\frac{s}{b}r)   $.  The scaling~(\ref{BetaFormb<s}) makes use of the linearly repelling fixed point of the map $M_{b,s}$ at $x=0$ with $M'(0)=s/b>1$ when $s>b$.  The point $x=0$ is an attractor when $s<b$, so it is nonsensical to look for a high-temperature scaling limit in that case.

The case $b=s$, for which  $x=0$ is marginally repelling for the map $M_{b,s}$, requires a more intricate scaling than the $b<s$ case.
\begin{itemize}
\item For $b\in \{2,3,4,\cdots\}$, let the constants $\kappa_b,\eta_b \in \R_+$ be defined as
$$ \kappa_{b}:=\sqrt{\frac{2}{b-1}}  \hspace{1.2cm} \text{and}\hspace{1.2cm} \eta_{b}:=\frac{b+1}{3(b-1) }\,.$$
\item For a fixed parameter value $r\in \R$, define $\big(\beta_{n, r}^{(b)}\big)_{n\in \mathbb{N}}$ to be a sequence in $\R_+$ with the $n\gg 1$ asymptotic form
\begin{align}\label{BetaForm}
\beta_{n, r}^{(b)}\, :=\,  \frac{\kappa_{b}}{\sqrt{n}}\,-\,\frac{ \tau\kappa_{b}^2 }{2n}\,+\,\frac{\kappa_{b}\eta_{b}\log n}{n^{\frac{3}{2}}}\,+\,\frac{\kappa_{b}r}{n^{\frac{3}{2}}}\, +\,\mathit{o}\Big( \frac{1}{n^{\frac{3}{2}}} \Big) \,,
\end{align}
where $  \tau :=\mathbb{E}[\omega_a^3]$ (the skew of $\omega_a$).  In this case,
 \begin{align}\label{VForm}
V_{\beta_{n,r}^{(b)}}\,:=\,\textup{Var}\Big(  W_{n}\big(\beta_{n,r}^{(b)}\big) \Big)\,=\, \textup{Var}\Bigg( \frac{ e^{ \beta_{n, r}^{(b)}\omega}}{\mathbb{E}\big[e^{\beta_{n,r}^{(b)}\omega}\big] }\Bigg) \,=\, \kappa_b^2  \bigg(\frac{ 1  }{n}   \,+\, \frac{  \eta_b \log n  }{n^2}  \,+\, \frac{  r}{n^2} \bigg)  \,+\,\mathit{o}\Big(\frac{1}{n^2}  \Big)  \,.
\end{align}

\end{itemize}
The $\tau$ term  is needed in~(\ref{BetaForm}) so that it disappears from~(\ref{VForm}).

With~(\ref{VForm}), the  following lemma implies that the variance, $\rho_{n}\big(\beta_{n,r}^{(b)}\big)  =M^{n}_b\big(V_{\beta_{n,r}^{(b)}}\big)  $, of $W_{n}(\beta_{n,r}^{b,s})$ converges as $n\rightarrow\infty$.  The proof is based on elementary estimates; see Section~\ref{SecVariance}.

\begin{lemma}\label{LemVar} Assume $s=b$, and  for a fixed value of the parameter $r\in \R$, let the sequence $(X^{(n,r)})_{n\in \mathbb{N}} $ satisfy the large $n$ asymptotics 
\begin{align}
 X^{(n,r)}  =  \kappa_{b}^2 \bigg(\frac{ 1 }{n}  +\frac{ \eta_{b}\log n}{n^2}+\frac{r}{n^2}\bigg)\,+\,\mathit{o}\Big(\frac{1}{n^2}  \Big) \, .\label{Assumption}
\end{align}
Then there exists a function $R_b:\R\rightarrow \R_+$ such that 
\begin{align}\label{LimitAssump}
   M^{n}_b\big(X^{(n,r)}   \big)   \,  \stackrel{n\rightarrow \infty  }{\longrightarrow} \, R_b(r) \, .
\end{align}
The convergence is uniform over  bounded intervals in $r\in \R$ provided that the error term $\mathit{o}(1/n^2  )$ in~(\ref{Assumption}), regarded as a function of $r$, is uniformly controlled over bounded intervals.        The  additional properties below hold for  the function $R_b$.
\begin{enumerate}[(I).]
\item Composition with the map $M_b$ translates the parameter $r$:
   $$ M_b\big(R_b(r)\big)\,=\, R_b(r+1) $$
\item  $R_b(r)$ is continuously differentiable with derivative satisfying
 $$ \frac{d}{dr}R_b(r)\,=\,\lim_{n\rightarrow \infty} \frac{\kappa_b^2}{n^2}\prod_{k=1}^n \Big(1+ R_b\big(r-k\big)\Big)^{b-1}\,. $$

\item As $r\rightarrow \infty$, $R_b(r)$ grows super-exponentially.  As $r\rightarrow -\infty$, $R_b(r)$ vanishes with the asymptotics 
$$  R_b(r)\,=\,-\frac{ \kappa_b^2 }{ r }\,+\, \frac{ \kappa_b^2\eta_b\log(-r) }{ r^2 }\,+\,\mathit{O}\Big( \frac{1}{r^3} \Big)  \,. $$
\end{enumerate}

\end{lemma}

 \begin{remark}\label{Remark} It is interesting that the initial variance scaling~(\ref{VForm}) involves three important terms rather than one as in~(\ref{XFormb<s}).   Since $\varrho_n\big(\beta_{n,r}^{(b)}\big)$ is convergent with large $n$,
  it follows that
   $$\varrho_n\bigg( \frac{t}{\sqrt{n}} \,+\,\mathit{o}\Big(\frac{1}{\sqrt{n}}\Big)  \bigg) \,=\,\begin{cases} 0   & \hspace{1cm}  0\leq t\leq \kappa_b \,, \\ \infty   & \hspace{1cm} \kappa_b < t  \,, \end{cases}\,  $$
and furthermore, if $\tau=0$, 
  $$\varrho_n\bigg( \frac{\kappa_b}{\sqrt{n}} \,+\,\frac{t \log n}{n^{\frac{3}{2}}} \,+\, \mathit{o}\Big(\frac{\log n}{n^{\frac{3}{2}}}\Big)    \bigg)\,=\,\begin{cases} 0   & \hspace{1cm}  0\leq t < \kappa_b\eta_b \,, \\ \infty   & \hspace{1cm} \kappa_b\eta_b < t  \,.\end{cases}  $$

 \end{remark}

\vspace{.4cm}

\noindent \textbf{(e). A theorem and a conjecture}\vspace{.2cm}\\
The following theorem is the main result of this article and concerns the moment behavior of the normalized partition function as $n\nearrow \infty$ when the inverse temperature is taken to be $\beta \equiv\beta_{n, r}^{(b)}$. Recall that the variables $W_{n}(\beta)$  have expectation $1$.

\begin{theorem}[Convergence of moments]\label{ThmMain}
Pick $b\in \{2,3,4,\cdots\}$ and assume $s=b$.    Define $Y_n^{(r)}=W_{n} \big(\beta_{n, r}^{(b)} \big)$ for $r\in \R$, $n\in \mathbb{N}$, and  $\beta_{n,r}^{(b)}$ as in~(\ref{BetaForm}).  As $n\rightarrow \infty$ the centered positive integer moments converge:
 $$ \mathbb{E}\Big[ \big(Y_n^{(r)}-1\big)^m\Big]\quad \longrightarrow \quad  R_{b}^{(m)}(r)\in \R_+\,.$$ 

The limiting moment functions  $R_{b}^{(m)}(r)$ satisfy the following properties:
\begin{enumerate}[(I).]

\item  For all $r\in \R$,
$$  R_{b}^{(2)}(r+1)\,=\,M_b\big(R_{b}^{(2)}(r)   \big)\,:=\, \frac{1}{b}\Big[\big(1+R_{b}^{(2)}(r)\big)^{b}-1\Big] \,. $$

\item  More generally, there exist  multivariate polynomials $P_m: \R^{m-1}\rightarrow \R$ of degree $b\cdot \min\big(b, \lfloor m/2\rfloor\big)$  with nonnegative coefficients such that for all $r\in \R$
$$ R_{b}^{(m)}(r+1)\,=\, P_m\big(R_{b}^{(2)}(r), R_{b}^{(3)}(r), \cdots, R_{b}^{(m)}(r)   \big)     \,. $$

\item  As $r\searrow -\infty$,
\begin{align*}
R_{b}^{(m)}(r)\,=\,\begin{cases}\kappa_b^{m}\frac{ m!  }{ 2^{\frac{m}{2}}(\frac{m}{2})! }|r|^{-\frac{m}{2}} \,+\, \mathit{O}\left(  |r|^{-\frac{m}{2}-1} \right)   & \quad \text{$m$ even,} \vspace{.2cm}  \\ \mathit{O}\Big(  |r|^{-\frac{m+1}{2}} \Big)    & \quad  \text{$m$ odd.}  \end{cases} 
\end{align*}
In words, the leading term of the the function $R_{b}^{(m)}(r)$ as $r\searrow -\infty$ agrees with the $m^{th}$ moment of a normal distribution with  mean zero and variance $\kappa_b^2/|r|$.

\end{enumerate}

\end{theorem}

\begin{remark}
The function $R_{b}^{(2)}(r)$ in the statement of Theorem~\ref{ThmMain} is the same as $R_{b}(r)$ in Lemma~\ref{LemEstimates}.

\end{remark}

For any fixed $r\in \R$ the sequence $\big\{R_{b}^{(m)}(r)\big\}_{m\in \mathbb{N}}$ of limiting moments grows at a rate significantly faster than the factorial of $n$, and thus the above theorem does not suffice to prove the following obvious conjecture.

\begin{conjecture}[Functional convergence]\label{Conj} Under the assumptions of Theorem~\ref{ThmMain}, there is convergence in law as $n\rightarrow \infty$
$$ W_{n}\big(\beta_{n,r}^{(b)}\big)\hspace{1cm}\Longrightarrow \hspace{1cm} L_{r}^{(b)}    $$
to a family of limit distributions $\big\{L_{r}^{(b)}\big\}_{r\in \R}$ satisfying the following:
\begin{enumerate}[(I).]

\item Let $\mathbf{W}_{r}$ be a random variable with distribution $L_{r}^{(b)}$. Then $\mathbf{W}_{r}$ has mean $1$ and the centered variables $  \sqrt{-r} (\mathbf{W}_{r}-1) $ converge in law as $r\rightarrow -\infty$  to a mean zero normal with variance $\kappa_b^2$.

\item  If $\mathbf{W}^{(i,j)}_r$ are independent variables with distribution $L_{r}^{(b)}$, then there is equality in distribution
\begin{align}\label{Simul}
\mathbf{W}_{r+1} \,\stackrel{d}{=}\,\frac{1}{b}\sum_{1\leq i\leq b}  \prod_{1\leq j\leq b} \mathbf{W}^{(i,j)}_r   \,.    
\end{align}

\end{enumerate}

\end{conjecture}

\subsection{Further discussion  }

The approach in articles~\cite{AKQ} and~\cite{CSZ1} of analyzing the  high-temperature behavior of the partition function through an expansion that converges term-by-term to a limiting chaos expansion is not applicable to the diamond lattice polymer in the case $s=b$.  I will argue this point through the following heuristic failure at determining a suitable high-temperature scaling $0<\beta_{n}^{(b)}\ll 1$ when $n\gg 1$: By linearizing the partition function, I get
\begin{align}\label{FakeNews}
W_{n}\big(\beta_{n}^{(b)} \big)\, := \,  \frac{1}{|\Gamma_{n}^{b,b}|  }\sum_{p\in \Gamma_{n}^{b,b}  }\prod_{a\Try p} \frac{e^{\beta_{n}^{(b)} \omega_a } }{ \mathbb{E}\big[e^{  \beta_{n}^{(b)}\omega_a  } \big]}
\,\stackrel{\beta_{n}^{(b)}\ll 1 }{ \approx  }\,& \frac{1}{|\Gamma_{n}^{b,b}|  }\sum_{p\in \Gamma_{n}^{b,b}  }\prod_{a\Try p} \big( 1\,+\, \beta_{n}^{(b)}\omega_a \big) 
\nonumber \\  \,=\,& 1\,+\,\sum_{k=1}^\infty \big( \beta_{n}^{(b)}\big)^{k}Q_{k}^{(n)}\big( \omega_{a};\, a\in E_{n}^{b,b}  \big)\,,
\end{align}
where $Q_{k}^{(n)}$ is a homogeneous degree $k$ polynomial of the disorder variables $ \omega_a$.  The linear term $k=1$ in the above expansion has the form
$$\beta_{n}^{(b)}  Q_{1}^{(n)}\big( \omega_{a};\, a\in E_{n}^{b,b}  \big) \,=\,\frac{\beta_{n}^{(b)}}{ b^n }\sum_{ a\in E_{n}^{b,b}} \omega_{a}\, $$
since the  probability that a randomly chosen path passes through any given bond $a\in E_{n}^{b,b}$  is   $1/b^{n}$.  The variance of the linear term is $(\beta_{n}^{(b)})^2$ since $\big|E_{n}^{b,b}\big|=b^{2n}$.  Thus, the linear term converges in probability to zero if  $\beta_{n}^{(b)}\searrow 0$ as $n\rightarrow \infty$.   Likewise, any other fixed term $k\geq 2$ from the expansion~(\ref{FakeNews}) will converge in probability to zero if $\beta_{n}^{(b)}\searrow 0$.  On the other hand, the variance of $ W_{n}(\beta_{n}^{(b)} )$ blows up as $n\rightarrow \infty$ if $\beta_{n}^{(b)} $ does not vanish quickly enough by Remark~\ref{Remark}.\vspace{.2cm}

\subsection{The rest of this article}

 I will assume that  $s=b$ throughout the remainder of this text.  The organization is as follows:  Section~\ref{SecVariance} contains the proof Lemma~\ref{LemVar} and the parts of Theorem~\ref{ThmMain} pertaining to the limiting variance of the variables $ W_n( \beta_{n, r}^{(b)} ) $. Sections~\ref{SecRecRel}-\ref{SecMainProof}  show how the convergence of the variances can be leveraged to obtain convergence of the higher centered moments.  In Section~\ref{SecGaussian} I prove part (III) of Theorem~\ref{ThmMain} regarding the  Gaussian behavior of the limiting moments $R^{(m)}(r)$ when $-r\gg 1$.

\section{Variance analysis}\label{SecVariance}

In this section, I will prove the results from Theorem~\ref{ThmMain} relevant to   the variance case ($m=2$), which are stated in the following corollary of Lemma~\ref{LemVar}.

\begin{corollary}\label{CorVar}
Let $\beta_{n, r}^{(b)}>0$ have the large $n$ asymptotics~(\ref{BetaForm}). Then the  variances $\textup{Var}\big(  W_n( \beta_{n, r}^{(b)} ) \big)$ converge as $n\rightarrow \infty$ to $R_b(r)$, where the function $R_b(r)$ satisfies the properties (I)-(III) listed in Lemma~\ref{LemVar}.   The convergence is uniform over bounded intervals.

\end{corollary}

The proof of Corollary~\ref{CorVar} is placed at the end of this section after the proof of Lemma~\ref{LemVar}.  Before I move to the proof of Lemma~\ref{LemVar}, I will develop some estimates on the polynomial map $M_b:[0,\infty)\rightarrow [0,\infty)$ having the form   $M_b(x)=((1+x)^b-1)/b $, or, more  precisely, estimates involving its function inverse $M_b^{-1}$.

\begin{lemma}\label{LemEstimates} Let $M_b^{-k}$ be the $k$-fold composition of the function inverse of $M_b$.

\begin{enumerate}[(i).]
\item The series $\displaystyle S_b(x):=\sum_{k=0}^{\infty}\big(M_b^{-k}(x)\big)^2$ is convergent for all $x\geq 0$.  Moreover, $S_b(x)$ is $\mathcal{O}(x)$ as $x\searrow 0$.

\item  For $n\gg 1$, $M_b^{-n}(x)= \kappa_b^2/n +\mathit{O}\big( 1/n^2\big)$.  Moreover, there is a $C>0$ such that for small enough $x\geq 0$ and all $n\in \mathbb{N}$
$$ \bigg| M_b^{-n}(x)\,-\,\frac{x}{1+\frac{n}{\kappa_b^{2}} x}  \bigg| \,\leq \,\frac{C\log\big(1+nx\big) }{n^2}\,. $$

\item   There is a $C>0$ such that for small enough $x\geq 0$ and all $n\in \mathbb{N}$
$$ \frac{d}{dx}M_b^{-n}(x)\,\leq \, \frac{C}{\big(1+\frac{n}{\kappa_b^2}x\big)^2} \,. $$

\item  The limit $\displaystyle D_b(x) :=  \lim_{n\rightarrow \infty} \frac{1}{n^2}\prod_{k=1}^n \Big(1+ M_b^{-k}(x)\Big)^{b-1}   $ exists for all $x\geq 0$, and the convergence is uniform over bounded intervals.

\end{enumerate}

\end{lemma}
\begin{proof}  I can assume that $x\geq 0$ is $\ll 1$ since the sequence $\big\{ M_b^{-k}(x) \big\}_{k=1}^{\infty}$ decreases at an exponential rate  as long as the terms $M_b^{-k}(x)$ remain above some cut-off $c>0$.  In other words, for any $c>0$ there is a $\lambda_c>0$ such that $M_b^{-1}(x)\leq e^{-\lambda_c}x$ for all $x>c$.  Note that for $0<x\ll 1$,
 \begin{align*}
 M_b^{-1}(x)\,=\,\big(1+bx    \big)^{\frac{1}{b}}\,-\,1  \,=\,x\,-\,\frac{x^2}{ \kappa_b^{2}  }\,+\,\mathit{O}\big(x^3\big)\,.  
 \end{align*}
Pick $a_{1}\in (0,\kappa_b^{-2 })$ and $ a_{2}\in ( \kappa_b^{-2 }  , \infty)$, and define $U_\downarrow(x)=\frac{x}{1+ a_{2}x }  $, $U_\uparrow(x)=\frac{x}{1+ a_{1}x }$.  For small enough $x\geq 0$,
\begin{align}\label{Obglidu}
U_\downarrow(x)    \,\leq \,M^{-1}_b(x)\,\leq  \, U_\uparrow(x)  \,.   
\end{align}
Functions of the form $U(x)=\frac{x}{1+ax}$, which happen to be  fractional linear transforms, are useful in the estimates below because the $k$-fold composition has the form 
\begin{align}\label{FLT}
U^k(x)=\frac{x}{1+kax}\,,\quad \text{and the derivative is}\quad \frac{d}{dx}U^k(x)=\frac{1}{(1+kax)^2}\,.
\end{align}

 \vspace{.25cm}

\noindent Part (i): For $x\geq 0$ small, I can apply~(\ref{Obglidu}) to get
$$ S_b(x)\,\leq \,\sum_{k=0}^{\infty}\big(U_\uparrow^{k}(x)\big)^2\, = \,\sum_{k=1}^{\infty} \frac{x^2}{\big(1+ ka_1 x \big)^2   }   \,\leq  \, \int_0^{\infty}\frac{x^2}{\big(1+ ra_1 x \big)^2   } dr\,=\,\frac{x}{a_1}   \,.   $$

\vspace{.25cm}

\noindent Part (ii): For $U_b(x):=\frac{x}{1+ \kappa_b^{-2}x }$, there is a $c>0$ such that for all $x\geq 0$
\begin{align}\label{Ergtido}
 \big| M_b^{-1}(x)\,-\,   U_b(x)\big|\,\leq \, cx^3   \,   .  
 \end{align}
By inserting a telescoping sum and applying the triangle inequality, I have that
\begin{align*}
\Big| M_b^{-n}(x)-U^{n}_b(x)   \Big|\,\leq \,& \sum_{k=0}^{n-1} \Big|U^{ n-k-1}_b\big(M_b^{-(k+1)}(x)\big) \,-\,  U^{n-k}_b\big(M_b^{-k}(x)\big)  \Big|\,. 
\intertext{As a consequence of Taylor's theorem, the above is bounded by } 
\leq \,&  \sum_{k=0}^{n-1} \Big| M_b^{-1}\big(M_b^{-k}(x)\big) \,-\, U_b\big(M^{-k}_b(x)\big) \Big|  \frac{d}{dy} U_b^{n-k-1}(y)\Big|_{y= M^{-k}_b(x) }\,.
 \intertext{The above also used that the derivative of $U_b^{n-k-1}$ is increasing, and thus maximized at the right endpoint of the interval $\big[M^{-k-1}_b(x), M^{-k}_b(x)    \big]$.    Inserting the expression~(\ref{FLT}) for  the derivative of $U_b^{n-k-1}$ and applying~(\ref{Ergtido}) yields} 
  \leq \,& c \sum_{k=0}^{n-1}\big(M_b^{-k}(x)\big)^3 \bigg(\frac{M_b^{-k}(x)}{ 1+\frac{n-k-1}{\kappa_b^2}M_b^{-k}(x)  }\bigg)^2 \frac{1}{\big(M_b^{-k}(x)\big)^2}
 \\ \leq \,& c \sum_{k=0}^{n-1} U_\uparrow^{k}(x) \bigg(\frac{U_\downarrow^{k}(x)}{1+\frac{n-k-1}{\kappa_b^2}U_\downarrow^{k}(x)  }\bigg)^2\, = \, c \sum_{k=0}^{n-1} \frac{x}{1+k a_1 x   }\frac{x^2}{\big(1+k a_2x+\frac{n-k-1}{\kappa_b^2}x\big)^2  } \,,
 \intertext{where the second inequality above employs~(\ref{Obglidu}) and that $\ell(x):=x/(1+ax)$ for $a>0$ is an increasing function on $[0,\infty)$.  Since   $a_2>\kappa_b^{-2}$, I can replace $a_2$ by $\kappa_b^{-2} $ in the above to get } 
    \leq  \,& \frac{c\kappa_b^4}{(n-1)^2} \sum_{k=0}^{n-1} \frac{x}{1+ka_1   x}
   \\ \leq  \,& \frac{c\kappa_b^4}{(n-1)^2} \bigg(x\,+\, \int_0^{n-1}\frac{x}{1+ra_{1}x}dr\bigg)
    \,=\, \frac{c \kappa_b^4}{a_1(n-1)^2}\Big(x\,+\, \log\big(1+xna_{1}   \big)\Big)\,,\intertext{and hence there is a $C>0$ large enough so such that for all $n\in \mathbb{N}$ and $x\in \R_+$} \, \leq \,&\frac{C }{n^2} \log\big(1+xn   \big)\,.
\end{align*}

\vspace{.5cm}

\noindent Part (iii):  By the chain rule,
\begin{align*}
\frac{d}{dx}M^{-n}_b(x)\,= \,&\prod_{k=0}^{n-1}\frac{d}{dy} M^{-1}_b(y)\Big|_{y= M^{-k}_b(x)}\,. 
\intertext{The derivative of $M_b^{-1}(x)$ is bounded by  $1-2x/\kappa_b^2+cx^2$ for some $c>0$ and small enough $x\geq 0$, so    }
\,\leq \,&\prod_{k=0}^{n-1} \left( 1\,-\,\frac{2}{\kappa_b^2}M^{-k}_b(x) \,+\,c\big(M^{-k}_b(x)\big)^2 \right)\\
\,\leq \,&\exp\left\{ -\frac{2}{\kappa_b^2}\sum_{k=0}^{n-1} M^{-k}_b(x)\,+\,c\sum_{k=0}^{n-1} \big( M^{-k}_b(x)\big)^2 \right\}\,.
\intertext{Applying the triangle inequality and using the definition of $S_b(x)$ allows us to write   }
\,\leq \,&\exp\left\{ -\frac{2}{\kappa_b^2}\sum_{k=0}^{n-1} \frac{x}{1+ \frac{k}{\kappa_b^2}x  }\,+\, \frac{2}{\kappa_b^2}\sum_{k=0}^{n-1}\bigg| \frac{x}{1+ \frac{k}{\kappa_b^2} x }- M^{-k}_b(x)   \bigg| \,+\,cS_b(x)  \right\}\,,
\intertext{which by parts (i) and (ii) is smaller than    }
\,\leq \,&C\exp\left\{ -\frac{2}{\kappa_b^2}\sum_{k=0}^{n-1} \frac{x}{1+ \frac{k}{\kappa_b^2}x  } \right\}
\,\leq \,C\exp\left\{ -\frac{2}{\kappa_b^2}\int_0^n dy \frac{x}{1+ \frac{y}{\kappa_b^2}x  }\right\}
\,= \,C \frac{1}{\big( 1 + x\frac{n}{\kappa_b^2}\big)^2 }\,
\end{align*}
for some $C>0$ and all $n$ and small enough $x$.

\vspace{.5cm}
\noindent Part (iv): I will need a few estimates:
\begin{enumerate}[(I).]
\item There is a $c>0$ such that for all  $k\in \mathbb{N}  $ and $x\geq 0$ in a bounded interval
 $$\bigg| \Big(\frac{k-1}{k}\Big)^2\Big(1+ M_b^{-k}(x)\Big)^{b-1}\,-\,1\bigg| \, \leq \, \frac{c}{k^2}\,.$$

\item  The expression
\begin{align}\label{ProdWell}
 \frac{1}{n^2}\prod_{k=1}^n \Big(1+ M_b^{-k}(x)\Big)^{b-1} 
 \end{align}
is uniformly bounded for $n\in \mathbb{N}$ and $x$ in a bounded interval. 

\end{enumerate}
Statement (I) follows since $M_b^{-k}(x)=\frac{\kappa_b^2}{k}+\mathit{O}\big(\frac{1}{k^2}\big)=\frac{2}{(b-1)k}+\mathit{O}\big(\frac{1}{k^2}\big)$ by part (ii).  For statement (II)
\begin{align*}
 \frac{1}{n^2}\prod_{k=1}^n \Big(1+ M_b^{-k}(x)\Big)^{b-1}\,= \,  \prod_{k=1}^n \Big(\frac{k-1}{k}\Big)^2\Big(1+ M_b^{-k}(x)\Big)^{b-1}\,\leq  \,  \prod_{k=1}^n\Big(1+\frac{c}{k^2}\Big) \,\leq \,\exp\bigg\{ \sum_{k=1}^{\infty}\frac{c}{k^2}  \bigg\} \,.  
 \end{align*}

Now I will proceed with the proof by showing that the sequence~(\ref{ProdWell}) is Cauchy.    For $N>n$, I can write
\begin{align*}
 \Bigg|  \frac{1}{N^2}&\prod_{k=1}^N \Big(1+ M_b^{-k}(x)\Big)^{b-1}  \,-\,  \frac{1}{n^2}\prod_{k=1}^n \Big(1+ M_b^{-k}(x)\Big)^{b-1}  \Bigg| \\ & \,= \,\frac{1}{n^2}\prod_{k=1}^n \Big(1+ M_b^{-k}(x)\Big)^{b-1}\Bigg|  \frac{n^2}{N^2}\prod_{k=n+1}^N \Big(1+ M_b^{-k}(x)\Big)^{b-1} \,-\, 1     \Bigg| \,.    
\intertext{Rewriting the right term above and applying (II) to the left term  gives us a $C>0$ such that  }   
 &\,\leq  \,C \Bigg|  \prod_{k=n+1}^N \Big(\frac{k-1}{k}\Big)^2\Big(1+ M_b^{-k}(x)\Big)^{b-1} \,-\, 1     \Bigg|\,.
\intertext{(I) and (II) give us the second inequality below:   }
 &\,\leq  \,C \sum_{k=n+1}^N \Bigg|\Big(\frac{k-1}{k}\Big)^2\Big(1+ M_b^{-k}(x)\Big)^{b-1} \,-\, 1     \Bigg| \prod_{\ell=k+1}^n \Big(\frac{\ell-1}{\ell}\Big)^2\Big(1+ M_b^{-\ell}(x)\Big)^{b-1} \,.  \\
  &\,\leq  \,C^2 \sum_{k=n+1}^N \frac{c}{k^2}\hspace{.5cm} \stackrel{n\rightarrow \infty}{\longrightarrow}\hspace{.5cm} 0  \,.
  \end{align*}
Hence the sequence~(\ref{ProdWell}) is convergent, and the convergence is uniform over bounded intervals.

\end{proof}

With the estimates in Lemma~\ref{LemEstimates} at my disposal, I next turn to the proof of Lemma~\ref{LemVar}.  Recall that $X^{(n,r)}$ is a sequence of the form~(\ref{Assumption}), where the error term $\mathit{o}(1/n^2)$ is uniformly controllable for  $r$ in bounded intervals $\mathcal{I}\subset \R$.  In other terms, for $n\gg 1$ and  any bounded $\mathcal{I}$ 
$$ \sup_{r\in \mathcal{I} } \bigg|X^{(n,r)}\,-\,\kappa_{b}^2 \bigg(\frac{ 1 }{n}  +\frac{ \eta_{b}\log n}{n^2}+\frac{r}{n^2}\bigg)       \bigg|\,=\mathit{o}\Big( \frac{1}{n^2}  \Big) \, . $$
\vspace{.2cm}

\begin{proof}[Proof of Lemma~\ref{LemVar}]  I will begin by proving the existence of the limit~(\ref{LimitAssump}).   It is sufficient for us to show that $M^{n}_b\big(X^{(n,r)}   \big)  $ converges with large $n$  to a limit, $R_b(r)$, for all $r$ smaller than some cut-off $r_0 \in \R$.  To demonstrate this claim, let us temporarily assume that $M^{n}_b\big(X^{(n,r)}   \big)\stackrel{n\rightarrow \infty}{\longrightarrow} R_b(r)  $ for $r\leq r_0$ and   pick  some $r'> r_0$.  Fix some  $k\in \mathbb{N}$ with  $k>r'-r_0$.  I can write
\begin{align}
M^{-k}_b\big(X^{(n,r')}   \big)\,=\,& M^{-k}_b \bigg(\frac{  \kappa_{b}^2 }{n}  +\frac{ \kappa_{b}^2 \eta_{b}\log n }{n^2}+\frac{ \kappa_{b}^2r'}{n^2}\,+\,\mathit{o}\Big(\frac{1}{n^2}   \Big) \bigg)   \,.\nonumber
\intertext{Moreover, through the quadratic approximation  $M^{-1}_b(x)=x- x^2/\kappa_b^2+\mathit{O}\big(x^3\big)$ for $0<x\ll 1$, I can expand $M^{-k}_b$ as  }
  \,=\,&\frac{  \kappa_{b}^2 }{n}  +\frac{ \kappa_{b}^2 \eta_{b}\log n }{n^2}+\frac{ \kappa_{b}^2(r'-k)}{n^2}\,+\,\mathit{o}\Big(\frac{1}{n^2}   \Big) \,.\label{Dish}
\end{align}
Using~(\ref{Dish}) and my assumption on the convergence of $M^{n}_b\big(X^{(n,r)} \big)$ for $r=r'-k<r_0$, 
\begin{align*}
M^{n}_b\big(X^{(n,r')} \big)\,=\,& M^k_b\bigg(M^{n}_b\Big(M^{-k}_b\big(X^{(n,r')}\big)   \Big)\bigg)\\ \,=\,& M^k_b\bigg(\underbrace{M^{n}_b\bigg(\frac{  \kappa_{b}^2 }{n}\,  +\,\frac{ \kappa_{b}^2 \eta_{b}\log  n }{n^2}\,+\,\frac{ \kappa_{b}^2(r'-k)}{n^2}\,+\,\mathit{o}\Big(\frac{1}{n^2}   \Big)  \bigg) }\bigg)\,\,\stackrel{n\rightarrow \infty}{\longrightarrow}\,\, M^k_b\big( R_b(r'-k)   \big) \,. \\   & \hspace{4cm} \stackrel{n\rightarrow \infty}{\longrightarrow }\,\,  R_b(r'-k) 
\end{align*}
Hence, $M^{n}_b\big(X^{(n,r')} \big)$ is convergent with large $n$.

Now I will prove that $M^{n}_b\big(X^{(n,r)}   \big)  $ converges with large $n$ for all $r\in \R$ sufficiently far in the negative direction. \vspace{.35cm}\\
\noindent \textbf{Defining inverted variables:} Let the sequence   $s_{j}^{(n,r)}\in (0,n)$ be defined by
$$s_{j}^{(n,r)} \, := \,   \frac{\kappa_{b}^2}{M^{j}_b\big(X^{(n,r)}  \big)}    \,. $$
Note that assumption~(\ref{Assumption}) is equivalent to
\begin{align}\label{Inuit}
s_{0}^{(n,r)}\,=\, n \,-\,  \eta_{b}\log n \,-\,r\,+\,\mathit{o}(1  )  \, .
\end{align}
I will identify $s_{j}\equiv s_{j}^{(n,r)}$ for  notational convenience  in the remainder of the proof.

The form of the map $M_b$  and  the definition of the $s_{j}$'s imply the recursive equation
\begin{align}\label{Selge}
  s_{j+1}     \,= &\,\bigg( \frac{1  }{   s_{j}  }\, +\, \frac{1  }{ s_{j}^2  }\,+\,\frac{2(b-2)}{3(b-1)} \frac{1  }{ s_{j}^3  }\,+\,\frac{1}{s_j^4}f_b\Big( \frac{1}{s_j}\Big)   \bigg)^{-1}   \,   \, ,
\end{align}
where $f_b(x):= \frac{1}{\kappa_b^2 x^4}\big(M_b\big(\kappa_b^2x\big)-\kappa_b^2x -\kappa_b^4x^2- \frac{(b-1)(b-2)}{6}\kappa_b^6x^3 \big)$.  Note that $f_b$ is a $(b-4)$-degree polynomial for $b\geq 4$ and $f_b=0$ for $b\leq 3$.  

Moreover, as long as $s_j>1$ I can write
\begin{align}\label{SelgeII}
s_{j}\,-\,1  \,= &\, \bigg(\frac{1  }{  s_{j}  }\, +\, \frac{1  }{   s_{j}^2  }\, +\, \frac{1  }{ s_{j}^3  }\,+\,\frac{1}{s_j^4}g\Big( \frac{1}{s_j}  \Big)   \bigg)^{-1}\, ,
\end{align}
where $g(x):=   \frac{ 1  }{1-x   } $.  Notice that the expressions within the inverses on the right sides of~(\ref{Selge}) and~(\ref{SelgeII}) differ only by
\begin{align}
\frac{y}{1-y}\,-\,   \frac{1}{\kappa_b^2} M_b\big(\kappa_b^2y  \big) \,=\,    \eta_b y^3\,+\,y^4g ( y  )\,-\,y^4f_b( y )\hspace{.4cm}\text{evaluated at}\hspace{.4cm}y=\frac{1}{s_j}\,,\label{PreHoink}
\end{align}
since $1-\frac{2(b-2)}{3(b-1)}=\frac{b+1}{3(b-1)}=:\eta_{b}$.  Note that~(\ref{PreHoink}) is positive because for any $k\in \mathbb{N}$
\begin{align*}
\frac{d^k}{dy^k} \frac{1}{\kappa_b^2} M_b\big(\kappa_b^2y  \big) \Big|_{y=0}\, =\, 1_{k\leq b} \frac{2^{k-1} (b-1)! }{(b-k)!(b-1)^{k-1}}\,\leq \,k!\,=\,\frac{d^k}{dy^k} \frac{y}{1-y}\Big|_{y=0}\,.
\end{align*}
Hence, $s_{j+1}>s_j-1  $.\vspace{.35cm}\\
\noindent \textbf{Controlling the difference between $\mathbf{s_{j+1}}$ and $\mathbf{s_{j}-1}$:}
Subtracting~(\ref{SelgeII}) from (\ref{Selge}), I get the second line below 
\begin{align}
s_{j+1}+1 -  s_{j}  \,=\,& \bigg(\frac{  \kappa_b^2  }{  M_b\big( \kappa_b^2y  \big)   }\,-\,\Big(\frac{1}{y}-1\Big)\bigg)\bigg|_{y=\frac{1}{s_j}}\label{PreDimple}
\\  \,=\, & \bigg( \Big( y + y^2+\frac{2(b-2)}{3(b-1)}y^3+y^4f_b( y)\Big)^{-1}\nonumber    \,-\,\Big( y + y^2+ y^3+y^4g( y) \Big)^{-1}\bigg)\bigg|_{y=\frac{1}{s_j}}\nonumber \\
\,=\,&y\bigg( \eta_b+ y\Big( g(y) - f_b(y)\Big) \bigg) \frac{ \kappa_b^2y (1-y)  }{ M_b\big( \kappa_b^2y \big)}\bigg|_{y=\frac{1}{s_j}}  \nonumber
 \\
\,=\,&\big(\eta_b y+y^2h_b(y)    \big)\Big|_{y=\frac{1}{s_j}} \,,\label{Dimple}
\end{align}
where the   function $h_b:[0,1]\rightarrow \R$ is defined by
\begin{align*}
h_b(x)\,=\,&  \eta_b x\bigg(  \frac{ \kappa_b^2x (1-x)  }{ M_b\big( \kappa_b^2x \big)}-1 \bigg)+ x\Big( g(x) - f_b(x)\Big)  \frac{ \kappa_b^2x (1-x)  }{ M_b\big( \kappa_b^2x \big)}\\ \,=\,&\frac{\kappa_b^2x}{ M_b\big( \kappa_b^2x \big)}\bigg(  \eta_b x(1-x) - \frac{ \eta_b}{\kappa_b^{2}}M_b\big( \kappa_b^2x \big) + x -   x(1-x)f_b(x)  \bigg)\,. 
  \end{align*}
 The above form shows that $h_b$ has a bounded, continuous derivative since $f_b$ and $M_b$ are polynomials and $M_b(x)/x$ is $\geq 1$ on $\R_+$.\vspace{.35cm}\\
\noindent \textbf{Estimating $\mathbf{s_{k}}$ through a series:}
Fix some $\epsilon\in (1,\infty) $ and define  $u_n^{(\epsilon)}\in \mathbb{N}$ as the first number $k=u_n$ such that $s_k  < \epsilon $.   For $ 1\leq k\leq  \min\big( n ,u_n^{(\epsilon)}\big)$, the value $s_k$ can be written in terms of~(\ref{Dimple}) and a telescoping sum as follows:
\begin{align}
s_k\,  = &\,n\, -\,r\, -\, \eta_{b}\log n \,+\,(s_k-s_0)\,+\,\mathit{o}(1)     \,\nonumber   \\
 =& \,n\,-\,r\, -\,    k  \,- \, \eta_{b}\log n \,+\,   \,\sum_{j=0}^{k-1}\big(s_{j+1}+1-s_{j}   \big)\,+\, \mathit{o}(1) \,\nonumber  \\
  = &\, n\,-\,r\, -\,  k  \,- \, \eta_{b}\log n \,+\, \eta_b\sum_{j=0}^{k-1}\frac{1}{s_{j}}  \, + \, \sum_{j=0}^{k-1}\frac{1}{s_{j}^2}h_b\Big(\frac{1}{s_j}  \Big)  \,+\, \mathit{o}(1) \, .\nonumber
 \intertext{Next I insert a telescoping sum of $\log s_k   $ terms:  }
 =&\, n\,-\,r \, -\,k  \,- \,\eta_{b}\log n\,+\,\eta_b(\log s_0-\log s_k      )  \,+\, \eta_b\sum_{j=0}^{k-1}\bigg(\frac{1}{s_{j}}-\log s_j+\log s_{j+1} \bigg)\nonumber \\
  &\,+\,\sum_{j=0}^{k-1}\frac{1}{s_{j}^2}h_b\Big(\frac{1}{s_j}\Big)     \,+\, \mathit{o}(1)\,, \nonumber  
\intertext{and since $\log n-\log s_0=\mathit{o}(1)$, which holds as a consequence of~(\ref{Inuit}), I also have} 
 =&\, n\,-\,r  \,- \, k \,- \,\eta_{b}\log s_k  \,+\, \eta_b\sum_{j=0}^{k-1}\bigg(\frac{1}{s_{j}}+\log\Big(1+\frac{s_{j+1}-s_j}{s_j}  \Big) \bigg)\nonumber   \\
  &\,+\,\sum_{j=0}^{k-1}\frac{1}{s_{j}^2}h_b\Big(\frac{1}{s_j}\Big)     \,+\, \mathit{o}(1) \,.\nonumber
  \intertext{Now I will use~(\ref{Dimple}) to  substitute for the difference $s_{j+1}-s_j$ inside the logarithm     }\,.
  =&\, n\,-\,r \, -\,  k \,-\,\eta_{b}\log s_k   \,+\, \eta_b\sum_{j=0}^{k-1}\bigg(\frac{1}{s_{j}}+\log\bigg(1-\frac{1}{s_j}+\frac{\eta_b}{s_j^2} +\frac{1}{s_{j}^3 }h_b\Big( \frac{1}{s_{j} } \Big)\bigg) \bigg)\nonumber \\ &\,+\,\sum_{j=0}^{k-1}\frac{1}{s_{j}^2}h_b\Big(\frac{1}{s_j}\Big)     \,+\, \mathit{o}(1) \,\nonumber  \\
 =&\, n\,-\,r \, - \, k  \,-\,\eta_{b}\log s_k \,+\,\sum_{j=0}^{k-1}\Big(\frac{\kappa_b^2}{s_{j}}\Big)^2\widehat{h}_b\Big(\frac{\kappa_b^2}{s_j}\Big)     \,+\, \mathit{o}(1) \, ,\label{Heffle}
\end{align} 
where $\widehat{h}_b:(0,\kappa_b^2]\rightarrow \R$ is defined by
\begin{align}
\widehat{h}_b(x)\,: = &\,  \frac{1}{ x^2}\bigg( \frac{x}{\kappa_b^2} -\log\Big( \frac{M_b(x)}{x}\Big)        \bigg)      \,+\,h_b\Big(\frac{x}{\kappa_b^2} \Big)   \nonumber  \\
\,= &\, \frac{1}{ x^2}\bigg(  \frac{x}{\kappa_b^2}  +\log\bigg(1-\frac{x}{\kappa_b^2} +\eta_b\frac{ x^2}{ \kappa_b^4  }+\frac{x^3}{ \kappa_b^6 }h_b\Big( \frac{x}{\kappa_b^2}\Big)\bigg)\bigg)\,+\,h_b\Big(\frac{x}{\kappa_b^2} \Big)\,.\nonumber 
\end{align}
The first line above implies that $\widehat{h}_b$ has a bounded, continuous derivative since the polynomial $M_b(x)/x$ is $\geq 1$ on $(0,\infty)$ and $h_b$  has a bounded, continuous derivative. \vspace{.35cm}\\
\noindent \textbf{Translating back to  the $\mathbf{M^k_b\big(X^{(n,r)}\big)} $ variables:} I can write~(\ref{Heffle}) in terms of the $M^j_b\big(X^{(n,r)}\big) $'s as
\begin{align}\label{Yaper}
\frac{\kappa_b^2}{ M^k_b\big(X^{(n,r)}\big) }  \,  =&\, n\,-\,r\,  -\,k  \,-\,\eta_{b}\log\bigg(\frac{\kappa_b^2}{M^k_b\big(X^{(n,r)}\big) }\bigg) \,+\,\sum_{j=0}^{k-1}\widehat{h}_b\Big(M^j_b\big(X^{(n,r)}\big)\Big)\Big(M^j_b\big(X^{(n,r)}\big) \Big)^2     \,+\, \mathit{o}(1) \,.\nonumber 
\intertext{For $\widetilde{X}_k^{(n,r)}:= M^k_b\big(X^{(n,r)}\big) $ I have }
 \, \frac{\kappa_b^2}{\widetilde{X}_k^{(n,r)} }  =&\, n \,-\,r \,-\,k  \,-\,\eta_{b}\log\bigg(\frac{\kappa_b^2}{\widetilde{X}_k^{(n,r)}}\bigg) \,+\,\sum_{\ell=1}^{k}\widehat{h}_b\Big(M^{-\ell}_b\big(\widetilde{X}_k^{(n,r)}\big)\Big)     \Big(M^{-\ell}_b\big( \widetilde{X}_k^{(n,r)} \big)\Big)^2\,+\, \mathit{o}(1) \,\nonumber
\\ \,  =&\, n\,-\,r \, -\,k \,-\,\eta_{b}\log\bigg(\frac{\kappa_b^2}{\widetilde{X}_k^{(n,r)}}\bigg) \,+\,F_b\big( \widetilde{X}_k^{(n,r)}\big)   \,+\, \mathit{o}(1) \,,
\end{align}
where $F_b:(0,\kappa_b^2]\rightarrow \R$ is defined by
$$F_b(x)\,:=\, \sum_{\ell=1}^{\infty}\widehat{h}_b\big(M^{-\ell}_b(x)\big)\big(M^{-\ell}_b( x)\big)^2 \,.  $$
Part (i) of Lemma~\ref{LemEstimates} and the boundedness of $\widehat{h}_b$ imply that the series defining $F_b(x)$ is absolutely convergent for all $x\in (0,\kappa_b^2]$ and that $F_b(x)=\mathit{O}(x)$ for $x\ll 1$.  In the last line of~(\ref{Yaper}), I was able to throw in the tail of the series since
\begin{align*}
 \sum_{\ell=k+1}^{\infty}\widehat{h}_b\Big(M^{-\ell}_b\big(\widetilde{X}_k^{(n,r)}\big)\Big)\Big(M^{-\ell}_b\big( \widetilde{X}_k^{(n,r)}\big)\Big)^2  \,=\,& \sum_{j=1}^{\infty}\widehat{h}_b\Big(M^{-j}_b\big( X^{(n,r)}\big)\Big)\Big(M^{-j}_b\big(  X^{(n,r)}\big)\Big)^2\\  \, = \,& F_b\big(  X^{(n,r)}\big)    \, = \,\mathit{O}\big( X^{(n,r)}\big)\, = \,\mathit{O}\Big(\frac{1}{n}\Big)\,.
 \end{align*}
Rearranging~(\ref{Yaper}),  I have
\begin{align}
 n \,-\,r \, -\,k  \,+\, \mathit{o}(1)\,  =& \,\frac{\kappa_b^2}{\widetilde{X}_k^{(n,r)} } \,+\,\eta_{b}\log\bigg(\frac{\kappa_b^2}{\widetilde{X}_k^{(n,r)}}\bigg)\,-\,F_b\big( \widetilde{X}_k^{(n,r)}\big) \,:=\,G_b\big(\widetilde{X}_k^{(n,r)}\big)   \,,\label{Ubdle}
\end{align}
where $G_b:(0,\kappa_b^2]\rightarrow \R$ has the $0<x\ll 1$ asymptotics
$$ G_b(x)\,=\,\frac{\kappa_b^2}{x}\,-\,\eta_{b}\log x\,+\,\eta_{b}\log\big(\kappa^2_b\big) \,+\,  \mathit{O}(x)\,.  $$
\vspace{.35cm}\\
\noindent \textbf{The function inverse of $\mathbf{G_b(x)}$:} I would like to solve for $\widetilde{X}_k^{(n,r)}$ in~(\ref{Ubdle}), but for this I need to show that the function inverse, $G_b^{-1}$, exists over an appropriate domain.  The function $F_b(x)$ has a bounded, continuous derivative, which can be seen through the inequalities
\begin{align}
\big|F_b'(x)\big|\,=\,&\left|2\sum_{k=1}^{\infty}\widehat{h}_b\big(M^{-k}_b(x)\big)M^{-k}_b( x)\frac{d}{dx}M^{-k}_b( x)  \,+\, \sum_{k=1}^{\infty}\widehat{h}_b'\big(M^{-k}_b(x)\big)\big(M^{-k}_b( x)\big)^2\right| \nonumber  \\
\leq \,&  2C\Big(\sup_{0 < x\leq \kappa_b^2} \big|\widehat{h}_b(x)\big|  \Big)\sum_{k=1}^{\infty} \frac{ M^{-k}_b( x)}{\big(1+x\frac{k}{\kappa_b^2}\big)^2  }\,+\,\Big(\sup_{0 < x\leq \kappa_b^2} \big|\widehat{h}_b'(x)\big|  \Big) S_b(x)\,,\nonumber 
\intertext{where $C>0$ comes through an application of part (iii) of Lemma~\ref{LemEstimates}.  As a consequence of part (ii) of Lemma~\ref{LemEstimates}, $M^{-k}_b( x)$ is bounded by a constant multiple of $x/\big(1+x\frac{k}{\kappa_b^2}\big)$.  Hence, there is a $\widehat{C}>0$ such that }
\leq \,&  \widehat{C}\sum_{k=1}^{\infty} \frac{ x}{\big(1+x\frac{k}{\kappa_b^2}\big)^3 }\,+\, \widehat{C} S_b(x)\nonumber  \\ \,\leq \,&\widehat{C}\int_0^\infty \frac{x}{(1+xt)^3}dt\,+\,\widehat{C}S_b(x)\,=\, \frac{1}{2}\widehat{C}+\widehat{C}S_b(x)  \,.
\end{align}
The sum in the first line above is viewed as a  Riemann lower bound for the integral.  The right side of the last line is bounded over bounded intervals and, in particular, $[0,1]$.

Note that $G_b$ is strictly decreasing over intervals $[0,\delta]$ for small enough $\delta>0$ since 
the derivative of $\kappa_b^2/x -\eta_{b}\log x$ blows up towards $-\infty$ for $0<x\ll 1$ and $F_b$ has a bounded derivative by the analysis above. Consequently, the function inverse $G^{-1}_b(y)$ exists and is differentiable on the interval $[  G_b(\delta)  ,\infty)$, and has  the $ y \gg 1$ asymptotics
\begin{align}\label{Heggazy}
 G^{-1}_b(y)\,=\, \frac{\kappa_b^2}{y} \,+\,\frac{\kappa_b^2\eta_{b}\log y}{y^2}\,+\,\mathit{O}\Big(\frac{1}{y^3}\Big)\,. 
 \end{align}
\text{}\vspace{.35cm}\\
\noindent \textbf{The limit of  $\mathbf{M^{n}_b\big(X^{(n,r)}   \big)}$ as $\mathbf{n\rightarrow \infty}$:} 
As long as $ \widetilde{X}_k^{(n,r)}$ is smaller than $\delta$ and $k\leq \min \big(n,u_{n}^{(\epsilon)}\big)$, then I can write
\begin{align}\label{Tizdu}
 G^{-1}_b\big(  n \,-\,r \, -\,k  \,+\, \mathit{o}(1)  \big)\,=\, \widetilde{X}_k^{(n,r)}\,. 
\end{align}
If  $-r>0$ is chosen to be sufficiently large, and then $n$ is chosen to be sufficiently large (possibly based on $r$), then I will have that
\begin{itemize}
\item  $s_k:=\frac{\kappa_b^2}{\widetilde{X}_k^{(n,r)}}$ remains above $\epsilon$ for all $ k\in [1, n]$ . Consequently $u_{n}^{(\epsilon)}\geq n$, and thus equation~(\ref{Yaper}) is valid for $k=n$.

\item $ \widetilde{X}_n^{(n,r)}$ will lie in the interval $[0,\delta]$, where $G_b$ is invertible.

\end{itemize}
Therefore, since $\widetilde{X}_n^{(n,r)}:=M^{n}_b\big(X^{(n,r)}   \big)$, applying~(\ref{Tizdu}) with $k=n$ completes the proof 
$$ \lim_{n\rightarrow \infty}M^{n}_b\big(X^{(n,r)}   \big) \,=\,G^{-1}_b\big( -r   + \mathit{o}(1)  \big)\,=\, G^{-1}_b(  -r    )\,. $$
The second equality uses that $G^{-1}_b$ is differentiable and hence continuous over its domain.   \vspace{.35cm}\\
\noindent \textbf{Properties (I)-(IV):} Now I will discuss the properties listed for the limit function $R_b(r)$.  The fact that $M_b\big(R_b(r)\big)=R_b(r+1)$ was implicitly already derived in the beginning of the proof, but I can clarify this idea as follows: 
\begin{align*}
M_b\big(R_b(r)\big)\,=\,M_b\Big(\lim_{n\rightarrow \infty} M^{n}_b\big(X^{(n,r)}  \big)  \Big)\,=\,&\lim_{n\rightarrow \infty}  M^{n+1}_b\big(X^{(n,r)} \big)\,=\,\lim_{n\rightarrow \infty}  M^{n}_b\big(M_b\big(X^{(n,r)} \big)  \big)\\  \,=\,&\lim_{n\rightarrow \infty}  M^{n}_b\bigg( \kappa_{b}^2 \bigg(\frac{ 1 }{n}  +\frac{ \eta_{b}\log n }{n^2}+\frac{r+1}{n^2}\bigg)\,+\,\mathit{o}\Big(\frac{1}{n^2}  \Big)   \bigg)\\  \,=\,& R_b(r+1)\,.
\end{align*}
The last line uses the convergence result that I have proved above.

To see property (II), first notice that the derivative of $R_{b,n}(r):=M_b^{n}\big(X^{(n,r)}\big)$ has the form
\begin{align*}
R_{b,n}^{ '}(r)\,:=\,\frac{d}{dr}M^n_b\big(X^{(n,r)}\big) \,=\,&\frac{\kappa_b^2}{n^2}\prod_{k=1}^n  M'_b\Big(M^{k-1}_b\big(X^{(n,r)}\big)    \Big)\\  \,=\,&\frac{\kappa_b^2}{n^2}\prod_{k=1}^n \Big(1+M^{k-1}_b\big(X^{(n,r)}\big)    \Big)^{b-1} \\ \,=\,& \widetilde{D}_b^{(n)}\Big(M_b^n\big(X^{(n,r)}\big)  \Big) \,=\, \widetilde{D}_b^{(n)}\big(R_{b,n}(r)  \big)\,,
\end{align*}
for $\widetilde{D}_b^{(n)}:\R_+\rightarrow \R_+  $ defined by 
$$ \widetilde{D}_b^{(n)}(x)\,:=\,\frac{\kappa_b^2}{n^2}\prod_{k=1}^n \Big(1+M^{-k}_b(x)    \Big)^{b-1}\,. $$
However, $\widetilde{D}_b^{(n)}$ converges uniformly over bounded intervals to a limit $D_b$  by part (III) of Lemma~\ref{LemEstimates}.  Hence, my convergence results as $n\rightarrow \infty$ can be summarized as follows:
\begin{itemize}
\item $R_{b,n}(r)$ converges uniformly over bounded intervals to a limit $R_b(r)$. 

\item $R_{b,n}^{ '}(r)\,=\,\widetilde{D}_b^{(n)}\big( R_{b,n}(r) \big) $ converges uniformly over bounded intervals  to $D_b\big(R_b(r)\big)   $.
 
\end{itemize}
Hence, 
$$R_{b,n}^{'}(r)\,=\, \lim_{N\rightarrow \infty} \frac{\kappa_b^2}{n^2}\prod_{k=1}^n \Big(1+ M^{-k}_b\big(R_b(r)\big)\Big)^{b-1}   \,=\,  \lim_{n\rightarrow \infty} \frac{\kappa_b^2}{n^2}\prod_{k=1}^n \Big(1+ R_b(r-k)\Big)^{b-1}     \, .  $$

For (III), the super-exponential growth of $R_b(r)$ as $r\rightarrow \infty$ follows easily from the identity (I).  The asymptotic behavior (IV) of $R_b(r)$ as $r\rightarrow -\infty$ follows from~(\ref{Heggazy}) since $R_b(r):=G_b^{-1}(-r)$.

\end{proof}

\begin{proof}[Proof of Corollary~\ref{CorVar}] Recall that $\varrho_{k}^{(2)}( \beta) :=  \textup{Var} \big( W_k( \beta )\big)$.  Then, $\varrho_{k}^{(2)}\big(\beta_{n, r}^{(b)}\big) = M^{k}_b\big( X^{(n,r)}   \big)$, where $X^{(n,r)}:=\varrho_{0}^{(2)}\big( \beta_{n, r}^{(b)}\big)$ has the large $n$ asymptotics
\begin{align*}
X^{(n,r)}  \,=\,\mathbb{E}\Bigg[\bigg(\frac{e^{\beta_{n, r}^{(b)}\omega}}{\mathbb{E}\big[  e^{\beta_{n, r}^{(b)}\omega} \big]   }-1   \bigg)^2 \Bigg]\,=\,  \kappa_{b}^2 \bigg(\frac{ 1 }{n}  +\frac{ \eta_{b}\log n }{n^2}+\frac{r}{n^2}\bigg)\,+\,\mathit{o}\Big(\frac{1}{n^2}  \Big) \,,
\end{align*}
where the error $\mathit{o}(1/n^2)$ is uniformly controlled over bounded intervals. Hence,  $ \textup{Var} \big( W_n( \beta_{n, r}^{(b)})\big)= M^{n}_b\big( X^{(n,r)}   \big)$ converges uniformly over bounded intervals to $R_b(r)$ by Lemma~\ref{LemVar}.

\end{proof}

\section{Convergence of the centered moments}\label{SecCentMom}

\subsection{Recursive relations  for higher moments  }\label{SecRecRel}

Recall that the partition function satisfies the distributional recursive  relation
\begin{align}
 W_{n+1}(\beta)-1\,\stackrel{d}{=}\, & \frac{1}{b}\sum_{i=1}^{b}\Bigg[ \prod_{j=1}^{b} W_{n}^{(i,j)}(\beta) \,-\,1\Bigg]\,, \nonumber
\intertext{where $W_{n}^{(i,j)}(\beta)$ are independent copies of $W_{n}(\beta)$.   The above can be written in terms of the centered variables $ R_{n}^{(i,j)}(\beta):= W_{n}^{(i,j)}(\beta)-1$ as}
 R_{n+1}(\beta) \,\stackrel{d}{=}\, & \frac{1}{b}\sum_{i=1}^{b} \sum_{A\subseteq \{1,\cdots,b\} } \prod_{j\in A} R_{n}^{(i,j)}(\beta)\,.\label{RecRel}
 \end{align}
Let $\varrho_{k}^{(m)}( \beta)$ denote the $m^{th}$ centered moment of $W_k( \beta )$.   Then~(\ref{RecRel}) implies that
\begin{align}
\varrho_{n+1}^{(m)}( \beta)\,=\, \frac{1}{b^m}\mathbb{E}\Bigg[\bigg(\sum_{i=1}^{b}\sum_{A\subseteq \{1,\cdots,b\} } \prod_{j\in A} R_{n}^{(i,j)}(\beta)\bigg)^m \Bigg] \, =  \, P_m\big( \varrho_{n}^{(2)}( \beta)  ,  \varrho_{n}^{(3)}( \beta) , \cdots , \varrho_{n}^{(m)}( \beta) \big)\label{RecHigherMom}
\end{align}
for a joint polynomial $P_m(y_2, \cdots, y_{m})$.  The variables $y_j$ are indexed to correspond to the $j^{th}$ centered moment in~(\ref{RecHigherMom}). The following proposition concerns the polynomials $P_m$.

\begin{proposition}\label{PropPoly}
The multivariate polynomial $P_m$ is of degree $d_b:=b\cdot \min (b,\lfloor m/2 \rfloor  )$ and satisfies the properties below.
\begin{enumerate}[(i).]

\item   $P_m(y_2,\cdots, y_{m})$ has no constant term, and its only linear term is $\frac{1}{b^{m-2}} y_{m}$.  In other terms, there exist polynomials $U_m:\R^{m-1}\rightarrow \R$, $V_m:\R^{m-2}\rightarrow \R$ 
such that
\begin{align}\label{PMRecur}
 P_m(y_2, \cdots, y_{m})\,=\,\frac{1}{b^{m-2}}y_{m}\,+\,y_{m}U_m(y_2,\cdots, y_{m})\,+\,V_m(y_2,\cdots, y_{m-1}) \, ,   
 \end{align}
where the polynomials $y_{m}U_m(y_2,\cdots, y_{m})$ and $V_m(y_2,\cdots,y_{m-1})$  have no constant or linear terms.

\item The polynomial  $V_m(y_2,\cdots,y_{m-1})$ is a linear combinations of monomials $y_{j_1}\cdots y_{j_\ell}$ 
for $1\leq \ell \leq d_b$ with 
$$j_1+\cdots +j_\ell  \geq  \begin{cases}m & \quad  \text{$m$ even,}  \\    m+1& \quad  \text{$m$ odd.}   \end{cases}$$
The polynomial $y_{m}U_m(y_2,\cdots, y_{m})$ is a linear combination of monomials with $j_1+\cdots +j_\ell\geq m+2$.

\item Suppose that for $0<x\ll 1$ there are constants $c_j \in \R$ such that the variables $y_{j}\equiv y_j(x)$ have the asymptotics 
 $$y_{j}\,=\,\begin{cases} c_{j}x^{\frac{j}{2}}+\mathit{O}\big(x^{\frac{j}{2}+1}\big)  & \quad \text{$j$ even,} \\ \mathit{O}\big(x^{\frac{j+1}{2}}\big)    &  \quad \text{$j$ odd.} \end{cases} $$
Then when $m$ is odd, $ V_m(y_2 ,     \cdots,  y_{m-1})$ is $ \mathit{O}\big(x^{\frac{1}{2}}\big) $, and when $m$  is even,
\begin{align*}
V_m(y_2 ,     \cdots,  y_{m-1})\,=\,&x^{\frac{m}{2} } \frac{1}{b^{m}}\frac{d^m}{dt^m}\Bigg( 1+\sum_{1\leq j<\frac{m}{2}}  c_{2j}\frac{t^{2j}}{(2j)!}  \Bigg)^{b^2}\bigg|_{t=0} \,+\,\mathit{O}(x)  \,, \text{ or equivalently,}\\
\,=\,&\frac{1}{b^{m}}\frac{d^m}{dt^m}\Bigg( 1+\sum_{1\leq j<\frac{m}{2}}  y_{2j}\frac{t^{2j}}{(2j)!}  \Bigg)^{b^2}\bigg|_{t=0} \,+\,\mathit{O}(x)  \,.
\end{align*}

\end{enumerate}

\end{proposition}

\begin{proof}Parts (i) and (ii) follow easily from the defining relation~(\ref{RecHigherMom}) since the random variables $ R_{n}^{(i,j)}(\beta)$ have mean zero.  For part (iii), the case of $m$ odd follows from part (ii).  For $m$ even, I have that as $x\ll 1$
\begin{align}\label{Gring}
V_m(y_2 ,     \cdots,  y_{m-1})\,=\,&\bigg(\text{Lin.\ comb.\ of monomials $\prod_{i=1}^{\ell} y_{j_i}$ with  $\sum_{i=1}^{\ell}  j_i=m$ and $j_i<m$ }\bigg)\,+\,\mathit{O}\big( x^{\frac{m}{2}+1} \big)\,,\nonumber  \\
y_mU_m(y_2 ,     \cdots,  y_{m})\,=\,&\mathit{O}\big( x^{\frac{m}{2}+1 } \big)\,.
\end{align}
 The error on the top line of~(\ref{Gring}) is $\mathit{O}\big( x^{\frac{m}{2}+1}\big)$ rather than $\mathit{O}\big( x^{\frac{m+1}{2}}\big)$ since a monomial with $j_1+\cdots + j_\ell=m+1$ must have at least one odd term $j_i$.  A similar statement can be made about $P_m$:
\begin{align}
P_m(y_2 ,     \cdots,  y_{m})\,=\,&\bigg(\text{Lin.\ comb.\ of monomials $\prod_{i=1}^{\ell} y_{j_i}$ with  $\sum_{i=1}^{\ell}  j_i=m$  }\bigg)\,+\,\mathit{O}\big( x^{\frac{m}{2}+1} \big)\nonumber  \\  \,=\,&\,x^{\frac{m}{2} }\frac{1}{b^{m}}\frac{d^m}{dt^m}\Bigg( 1+\sum_{1\leq j\leq \frac{m}{2}} c_{2j}\frac{t^{2j}}{(2j)!}  \Bigg)^{b^2}\Bigg|_{t=0} \,+\,\mathit{O}\big( x^{\frac{m}{2}+1}\big)\,.\label{Nork}
\end{align}
The second inequality follows from~(\ref{RecHigherMom}) since there are $b^2$ random variables $R_n^{(i,j)}$ involved in the expression and only the  sets $A \subseteq \{1,\cdots, b\}$ containing a single element can give rise to the terms $y_{j_1}\cdots y_{j_\ell}$ with $j_1+\cdots + j_\ell=m$.   The generating function in~(\ref{Nork}) captures the relevant combinatorics for the coefficient of the leading term, $x^{\frac{m}{2}}$.  The above considerations give us the $x\ll 1$ asymptotics
\begin{align*}
V_m(y_2 ,     \cdots,  y_{m-1})\,=\,& P_m(y_2 ,     \cdots,  y_{m})\,-\,\frac{1}{b^{m-2}}y_m\,-\, y_{m}U_m(y_2 ,     \cdots,  y_{m}) \\
\,=\,& x^{\frac{m}{2} }\frac{1}{b^{m}}\frac{d^m}{dt^m}\Bigg( 1+\sum_{1\leq j\leq \frac{m}{2}} c_{2j}\frac{t^{2j}}{(2j)!}  \Bigg)^{b^2}\Bigg|_{t=0}\,-\,\frac{1}{b^{m-2}}c_m x^{\frac{m}{2}}
\,+\,\mathit{O}\big( x^{\frac{m}{2}+1}\big)\,\\
\,=\,& x^{\frac{m}{2} }\frac{1}{b^{m}}\frac{d^m}{dt^m}\Bigg( 1+\sum_{1\leq j < \frac{m}{2}} c_{2j}\frac{t^{2j}}{(2j)!}  \Bigg)^{b^2}\Bigg|_{t=0}
\,+\,\mathit{O}\big( x^{\frac{m}{2}+1}\big)\,.
\end{align*}

\end{proof}

\vspace{.3cm}

\noindent{\textbf{Notation and conventions:}}
\begin{itemize}

\item  For any $2\leq k<m$, I will interpret $P_k$ flexibly as a polynomial in $y_2$, $\dots$, $y_{m}$ that is independent of the variables $y_j$ for $j>k$:
$$ P_k(y_2,\cdots, y_{k})\,\equiv\,  P_k(y_2,\cdots, y_{m})    \,.   $$

\item For $x\in \R$ and $\mathbf{y}:=(y_3,\cdots, y_{m}) $, I define the vector-valued function $\vec{P}_m:\R^{m-1}\rightarrow \R^{m-2}$  
\begin{align*}
 \vec{P}_m (x,\mathbf{y}) \,:=\,\big(P_3(x,\mathbf{y}),\cdots, P_{m}(x,\mathbf{y})   \big) \,. 
 \end{align*}

\item   I define $\big(\mathbf{\widetilde{D}}\vec{P}_m\big) (x,\mathbf{y}) $ as the $b-2$ by $b-2$ matrix of partial derivatives  
$$\Big[\big(\mathbf{\widetilde{D}}\vec{P}_m\big) (x,\mathbf{y})\Big]_{i,j}\,=\, \frac{\partial}{\partial y_j} P_{i}(x,\mathbf{y})   \,$$
for $i,j\in \{3,4,\cdots, b\}$. I denote the partial derivative with respect to $x$ as $(\partial_1\vec{P}_m) (x,\mathbf{y})$.

\end{itemize}

\subsection{Iterating the recursive relation}

The following technical lemma shows how iterating the recursive relation~(\ref{RecHigherMom}) under a scaling defined through the variance maps $M_b$ yields a limit that is independent of the higher ($m\geq 3$) initial centered moments.   For a vector $\textbf{y}$, define $\|\textbf{y}\|_\infty :=\max_i |\textbf{y}_i|$, i.e., the max norm of  $\textbf{y}$.

\begin{lemma}\label{LemFunGeneral}
For $m\geq 3$ and $x\geq 0$, define $F_{n}^{(x)}:\R^{m-2}\rightarrow \R^{m-2}$ as 
$$F_{n}^{(x)}(\mathbf{y}):= \vec{P}_m \big(M^{-n}_b(x),\mathbf{y}\big)\,. $$

\begin{enumerate}[(i).]
\item For sufficiently small $\epsilon>0$ and all $  \| \mathbf{y} \|_\infty\leq \epsilon$, the limit
$$ \lim_{n\rightarrow \infty} F_{1}^{(x)}\circ\cdots     \circ F_{n}^{(x)}(\mathbf{y})\,=\,\big(H_3(x),  H_4(x),\cdots, H_m(x)\big)\,=:\,\vec{H}_m(x)$$
exists and is independent of the argument $\mathbf{y} \in \R^{m-2}$. Moreover, the convergence is uniform for all $ \|\mathbf{y}\|_{\infty}\leq \epsilon$ and $x$ in any bounded interval.

\item $H_m(x)$ is nonnegative and increasing on $[0,\infty)$ with $H_m(0)=0$ and $H_m(x)\nearrow \infty $ as $x\nearrow \infty$.

\item  The limit vector $\vec{H}_m(x)$ satisfies the recursion relation:
 $$ \vec{H}_m\big(M_b(x)\big)\,=\,\vec{P}_m \big(x,\vec{H}_m(x)\big)\,.  $$

\item The derivative of $\vec{H}_m(x)$ has the form
\begin{align*}
\frac{d}{dx}&\vec{H}_m(x)\\ &\,=\,\sum_{k=1}^{\infty} \Bigg( \prod_{j=1}^{k-1} (\widetilde{\mathbf{D}} \vec{P}_m)\Big(  M^{-j}_b(x)    , \vec{H}_m\big( M^{-j}_b(x)  \big)    \Big) \Bigg) (\partial_1 \vec{P}_m)\Big(M^{-k}_b(x), \vec{H}_m\big(M^{-k}_b(x)\big)\Big) \frac{d}{dx}M^{-k}_b(x)\,.
\end{align*}
The convergence of the series is uniform for $x\geq 0$ in bounded intervals.  In the above, I interpret a product $\prod_{i=1}^{n} A_i$ of square matrices $A_i$ as $A_1 A_2\cdots A_n$.  

\item For the $\epsilon>0$ from part (i), there is uniform convergence  for all  $\|\mathbf{y}\|_{\infty}\leq \epsilon$ and $x\geq 0$ in  any bounded interval:
$$ 
  \lim_{n\rightarrow \infty}\frac{d}{dx} F_{1}^{(x)}\circ\cdots     \circ F_{n}^{(x)}(\mathbf{y})\,=\,\frac{d}{dx}\vec{H}_m(x)\,.$$

\end{enumerate}

\end{lemma}

\begin{proof} As a consequence of part (i) of by Proposition~\ref{PropPoly}, the matrix of partial derivatives $\mathbf{D}F_{n}^{(x)}$ has entries $3\leq i,j\leq m$:
\begin{align}
\big[\mathbf{D}F_{n}^{(x)}(\mathbf{y})\big]_{i,j}\,=\,\frac{\partial}{\partial y_j}P_{i}\big(M^{-n}_b(x),y_3 ,\cdots ,y_{i}\big) \,,\nonumber
\end{align}
satisfying
\begin{align*}
\big[\mathbf{D}F_{n}^{(x)}\big]_{i,j}(\mathbf{y})\,=\,\begin{cases} 0  &\quad  i< j\,,  \\  \frac{1}{b^{i-2}} &\quad i=j\,, \\ \text{polynomial in $x$, $y_3$,$\cdots$, $y_i$ with no constant term} &\quad i>j \,. \end{cases}
\end{align*}
Thus when $x$, $\|\mathbf{y}\|_\infty$ are small, $\mathbf{D}F_{n}^{(x)}(\mathbf{y})$ is close to the diagonal matrix $ \big[\frac{1}{b^{i-2}}\delta_{i,j}\big]$.

Define the sequence $s_{k}^{(n)}\in \R^{m-2}$ as $s_0^{(n)}=\mathbf{y}$ and 
$$s_{k}^{(n)}\,=\,  F_{n-k+1}^{(x)}\circ\cdots     \circ F_{n}^{(x)}(\mathbf{y}) $$
for $1\leq k\leq n$.  For a square matrix $A$, let $\|A\|:=\max_{i}\sum_j|A_{i,j}| $ denote the operator norm of $A$ with respect to the max norm.
 By part (i) of Proposition~\ref{PropPoly}, the $s_{k}^{(n)}$'s satisfy the recurrence relation
\begin{align}\label{SRecur}
\big[s_{k+1}^{(n)}\big]_{i}\,=\,\frac{1}{b^{m-2}}\big[s_k^{(n)}\big]_{i} \,+\,  \big[s_k^{(n)}\big]_{i}  U_i\Big( M^{-(n-k)}_b(x) , s_{k}^{(n)} \Big)\,+\,V_i\Big( M^{-(n-k)}_b(x), s_{k}^{(n)} \Big)  \end{align}
for $0\leq k<n$.  Pick $\epsilon>0$ small enough so  that for all $i\in\{3,\cdots,m\}$
\begin{align}
\sup_{0<\|\mathbf{y}\|_\infty\leq\epsilon}  \big| &V_i(\mathbf{y} )\big|\,\leq \,\frac{\epsilon}{b^{i-2}} \,,\hspace{1cm}\text{and}\hspace{1cm}  \sup_{0<\|\mathbf{y}\|_\infty\leq \epsilon} \big| U_{i}(\mathbf{y})\big|\,\leq \,\frac{b-1}{2b^{i-2}}\,,\nonumber \\ &\sup_{ x, \|\mathbf{y}\|_\infty\leq\epsilon}  \Big\| \mathbf{D}F_{n}^{(x)}(\mathbf{y})\,-\,\Big[\frac{1}{b^{i-2}}\delta_{i,j}\Big]\Big\|  \,\leq \,\frac{b-1}{2b}    \,.  \label{DerDef}
\end{align}
If $M^{-(n-k)}_b(x)\leq \epsilon$ and $\|s_{k}^{(n)}\|_{\infty}\leq \epsilon$, then  the recursion relation~(\ref{SRecur}) implies that $\|s_{k+1}^{(n)}\|_{\infty}\leq \epsilon$.  Therefore, by induction, if $x,\|\mathbf{y}\|_\infty\leq \epsilon$, then the $\|s_{k}^{(n)}\|_\infty$'s stay below $\epsilon$ for all $k\in \mathbb{N}$.  

\vspace{.3cm}

\noindent Part (i): Let $\|\mathbf{y}\|_{\infty}\leq \epsilon$ for $\epsilon>0$ chosen above.  Also, I will temporarily assume that $ x\leq \epsilon$.  To see that the sequence $F_{1}^{(x)}\circ\cdots     \circ F_{n}^{(x)}(\mathbf{y}) $ is Cauchy, first notice that for $N>n$ I can write
\begin{align*}
F_{1}^{(x)}\circ\cdots     \circ F_{N}^{(x)}(\mathbf{y})\,=\,F_{1}^{(x)}\circ\cdots     \circ F_{n}^{(x)}\big(s_{N-n}^{(N)}\big)\,.
\end{align*}
By the remarks above, since $x,\|\mathbf{y}\|_\infty\leq \epsilon$, I have that $\|s_{N-n}^{(N)}\|_\infty \leq  \epsilon $, and thus  the first inequality below.
\begin{align*}
\Big\| F_{1}^{(x)}\circ\cdots     \circ F_{N}^{(x)}(\mathbf{y})\,-\, F_{1}^{(x)}\circ\cdots     \circ F_{n}^{(x)}(\mathbf{y})  \Big\|_{\infty}\,= \,&\Big\|F_{1}^{(x)}\circ\cdots     \circ F_{n}^{(x)}\big(s_{N-n}^{(N)}\big)\,-\,F_{1}^{(x)}\circ\cdots     \circ F_{n}^{(x)}(\mathbf{y}) \Big\|_{\infty}\\
\,\leq \,&2\epsilon\sup_{ x, \|\mathbf{y}\|_{\infty}\leq \epsilon}\big\|\mathbf{D} F_{1}^{(x)}\circ\cdots     \circ F_{n}^{(x)}(\mathbf{y})\big\|    \,\\
\,\leq \,&2\epsilon\Big(\frac{b+1}{2b}  \Big)^n
\end{align*}
To see the second inequality above, first notice that by the chain rule
\begin{align*}
 \mathbf{D} F_{1}^{(x)}\circ\cdots     \circ F_{n}^{(x)}(\mathbf{y})\,=\, \prod_{k=1}^{n} \big(\mathbf{D}F_k^{(x)}\big)\big( F_{k+1}^{(x)}\circ\cdots    \circ  F_{n}^{(x)}(\mathbf{y})\big)    \,=\,\prod_{k=1}^{n} \big(\mathbf{D}F_k^{(x)}\big)(s_{n-k}^{(n)}) \,.
 \end{align*}
Moreover, since $x,\|s_{n-k}^{(n)}\|_\infty\leq \epsilon$, the inequality~(\ref{DerDef}) implies
\begin{align}\label{Blim}
\big\|\mathbf{D}  F_{1}^{(x)}\circ\cdots     \circ F_{n}^{(x)}(\mathbf{y})\big\|\,\leq \, \Big(\frac{b+1}{2b}  \Big)^n\,.
\end{align}
Therefore, the statement of (i) holds for $x\leq \epsilon$.

To extend the analysis to $x>\epsilon$, let $\ell\in \mathbb{N}$ be the smallest value such that $M_b^{-\ell}(x)\leq \epsilon $.  Then,
\begin{align}\label{Hezel}
\vec{H}_m(x)\,:=\, \lim_{n\rightarrow \infty}F_{1}^{(x)}\circ\cdots     \circ F_{n}^{(x)}(\mathbf{y})\,=\, &  \lim_{n\rightarrow \infty}F_{1}^{(x)}\circ\cdots     \circ F_{\ell}^{(x)}\Big(F_{\ell+1}^{(x)}\circ\cdots     \circ F_{n}^{(x)}(\mathbf{y})\Big)\nonumber \\ \,=\,& \lim_{n\rightarrow \infty} F_{1}^{(x)}\circ\cdots     \circ F_{\ell}^{(x)}\Big( F_{1}^{(M_b^{-\ell}(x))}\circ\cdots     \circ F_{n-\ell}^{(M_b^{-\ell}(x))}(\mathbf{y})\Big) \,\nonumber
 \\ \,=\,&F_{1}^{(x)}\circ\cdots     \circ F_{\ell}^{(x)}\Big( \lim_{n\rightarrow \infty}  F_{1}^{(M_b^{-\ell}(x))}\circ\cdots     \circ F_{n-\ell}^{(M_b^{-\ell}(x))}(\mathbf{y})\Big) \nonumber
 \\ \,=\,&F_{1}^{(x)}\circ\cdots     \circ F_{\ell}^{(x)}\big(\vec{H}_m\big(M_b^{-\ell}(x)\big)\big) \,.
 \end{align}
The limit on the third line exists by the result above.

\vspace{.3cm}

\noindent Part (ii): The components of $\vec{H}_m(x)$ are nonnegative, increasing functions since the multivariate polynomials $P_j(x,\mathbf{y})$ defining the functions  $F_{n}^{(x)}$, $n\in \mathbb{N}$ have nonnegative coefficients.

\vspace{.3cm}

\noindent Part (iii): The identity $\vec{H}_m\big(M_b(x)\big)=\vec{P}_m\big(x,\vec{H}_m(x)\big)$ follows from~(\ref{Hezel}) by setting $\ell=1$ and replacing $x$ by $M_{b}(x)$.

\vspace{.3cm}

\noindent Parts (iv) and (v): For $\epsilon>0$ as above, let us assume $\|\textbf{y}\|_\infty\leq \epsilon $.   By the chain rule,
\begin{align}\label{Heurashna}
 \frac{d}{dx} F_{1}^{(x)}\circ\cdots     \circ F_{n}^{(x)}(\mathbf{y})\,=\,&\sum_{k=1}^{n}\mathbf{D}F_{1}^{(x)}\circ\cdots     \circ F_{k-1}^{(x)}(z)\Big|_{z = F_{k}^{(x)}\circ\cdots     \circ F_{n}^{(x)}(\mathbf{y})  }\nonumber \\  &\, \times (\partial_1 \vec{P}_m)\Big(M^{-k}_b(x),  F_{k+1}^{(x)}\circ\cdots     \circ F_{n}^{(x)}(\mathbf{y}) \Big)\frac{d}{dx}M^{-k}_b(x)\,.
\end{align}
For $1<k<n$,  part (i) implies the second equality below:
\begin{align*}
\lim_{n\rightarrow \infty} F_{k+1}^{(x)}\circ\cdots     \circ F_{n}^{(x)}(\mathbf{y})\,=\,\lim_{n\rightarrow \infty}  F_{1}^{(M^{-k}_b(x))}\circ\cdots     \circ F_{n-k}^{(M^{-k}_b(x))}(\mathbf{y})\,=\,\vec{H}_m\big( M^{-k}_b(x) \big)\,.
\end{align*}
Thus, a single term from the sum~(\ref{Heurashna}) has the $n\rightarrow \infty$ limit
\begin{align}
 \lim_{n\rightarrow \infty}&\mathbf{D}F_{1}^{(x)}\circ\cdots     \circ F_{k-1}^{(x)}(z)\Big|_{z = F_{k}^{(x)}\circ\cdots     \circ F_{n}^{(x)}(\mathbf{y})  }  (\partial_1 \vec{P}_m)\Big(M^{-k}_b(x), F_{k+1}^{(x)}\circ\cdots     \circ F_{n}^{(x)}(\mathbf{y}) \Big)\frac{d}{dx}M^{-k}_b(x)\nonumber  \\  &\, = \,\mathbf{D}F_{1}^{(x)}\circ\cdots     \circ F_{k-1}^{(x)}(z)\Big|_{z = \vec{H}_m\big( M^{-(k-1)}_b(x) \big)  }  (\partial_1 \vec{P}_m)\Big(M^{-k}_b(x), \vec{H}_m\big( M^{-k}_b(x) \big)\Big)\frac{d}{dx}M^{-k}_b(x)\,.\label{Dintle}
 \end{align}
The convergence is also uniform for $\|\mathbf{y}\|_{\infty}\leq\epsilon$ and $x\geq 0$ in any bounded interval since  $\mathbf{D}F_{1}^{(x)}\circ\cdots     \circ F_{k-1}^{(x)}(z)$ and $\partial_1 \vec{P}_m$ are vectors of  polynomials, and thus uniformly continuous in bounded regions of $(x,\mathbf{y})\in \R^{n-1}$. 

Next I will obtain a  bound for a single term from the sum~(\ref{Heurashna}).  As in the proof of part (i), I will temporarily assume $x\in[0,\epsilon]$.      A single  term has the bound
\begin{align}\label{Quatible}
\bigg\|\big(\mathbf{D}F_{1}^{(x)}&\circ\cdots     \circ F_{k-1}^{(x)}\big)(z)\Big|_{z = F_{k}^{(x)}\circ\cdots     \circ F_n^{(x)}(\mathbf{y})  }  (\partial_1 \vec{P}_m)\Big(M^{-k}_b(x), F_{k+1}^{(x)}\circ\cdots     \circ F_n^{(x)}(\mathbf{y}) \Big)\frac{d}{dx}M^{-k}_b(x)\bigg\|_\infty\nonumber \\  &\, \leq \,c\Big(\frac{b+1}{2b}  \Big)^k\,   \frac{1}{\big(1+k\frac{x}{\kappa_b^2}\big)^2} 
\end{align}
for some $c>0$.  The above uses~(\ref{Blim}) to bound the first term on the left side and part (iii) of Lemma~\ref{LemEstimates} to bound the rightmost scalar term, $\frac{d}{dx}M^{-k}_b(x)$,  by a multiple of $ \big(1+k\frac{x}{\kappa_b^2}\big)^{-2}$.  For the second term I have merely used that $ (\partial_1 \vec{P}_m)(x, \mathbf{y} )$ is bounded  for all $x,\|\mathbf{y}\|_\infty\leq \epsilon$.  It follows from~(\ref{Quatible}) that the sum~(\ref{Heurashna}) is uniformly convergent for $x,\|\mathbf{y}\|_\infty\leq \epsilon$ as $n\rightarrow \infty$.  Therefore, if $x, \|\textbf{y}\|_{\infty}\leq \epsilon$, I have
\begin{align*}
\lim_{n\rightarrow \infty} \frac{d}{dx} F_{1}^{(x)}\circ\cdots     \circ F_{n}^{(x)}(\mathbf{y})\,=\,&\sum_{k=1}^{\infty}  \prod_{j=1}^{k-1} (\widetilde{\mathbf{D}} \vec{P}_m)\Big(  M^{-j}_b(x)    , \vec{H}_m\big( M^{-j}_b(x)  \big)    \Big) \\ &\times(\partial_1 \vec{P}_m)\Big(M^{-k}_b(x),\vec{H}_m\big(M^{-k}_b(x)\big)\Big) \frac{d}{dx}M^{-k}_b(x)\,.
\end{align*}
Since the convergence above is uniform over bounded intervals, it follows that the above is equal to $\frac{d}{dx}\vec{H}_m(x)$.

The results can be extended from $x\leq \epsilon$ to all of $[0,\infty)$ by applying the identity in part (ii) and the fact that $\cup_j M_b^j([0,\epsilon])=[0,\infty)$.

\end{proof}

\subsection{Proof of Theorem~\ref{ThmMain} }\label{SecMainProof}

Now I am prepared to prove the results stated in Theorem~\ref{ThmMain} except for property (III), which I postpone to the next section. The key steps in the proof involve applications of Lemma~\ref{LemVar} and Lemma~\ref{LemFunGeneral}.
\begin{theorem}\label{ThmMomConv}
Let $\beta_{n,r}^{(b)}$ be defined as in~(\ref{BetaForm}). 
\begin{enumerate}[(i).]

\item  For any $m\geq 2$ and $r\in \R$, there is convergence as $n\rightarrow \infty$ of the moments
$$ \varrho_{n}^{(m)}\big( \beta_{n,r}^{(b)}\big) \quad \longrightarrow \quad  R_b^{(m)}(r)   \, .  $$

\item  The limit functions $R_b^{(m)}(r)$ are positive, increasing, and  satisfy 
$$\lim_{r\rightarrow -\infty}R_b^{(m)}(r)\,=\,0  \hspace{.5cm}\text{and}  \hspace{.5cm} \lim_{r\rightarrow \infty}R_b^{(m)}(r)\,=\,\infty \,.   $$

\item For all $r\in \R$,
   $$ P_m\Big(R_b^{(2)}(r),\cdots , R_b^{(m)}(r)\Big)\,=\, R_b^{(m)}(r+1) \,.$$

\item  For $U_m$ and $V_m$ defined as in Proposition~\ref{PropPoly},
\begin{align*}
 R_b^{(m)}(r)\,=\, \sum_{k=1}^{\infty}&V_m\Big(R_b^{(2)}(r-k),\cdots, R_{b}^{(m-1)}(r-k)\Big)\\  &  \times \prod_{j=1}^{k-1}\bigg(\frac{1}{b^{m-2}}\,+\, U_m\Big(R_b^{(2)}(r-j),\cdots, R_b^{(m)}(r-j)\Big)\bigg)\,.
\end{align*}

\item  The derivative of $R_b^{(m)}(r)$ given by the last component of the vector equation 
\begin{align*}
\big(R_b^{(2)\,'}(r),\cdots, R_b^{(m)\,'}(r)\big)\,=\,&\sum_{k=1}^{\infty} \prod_{j=1}^{k-1} (\mathbf{\widetilde{D}} \vec{P}_m)\Big( R_b^{(2)}(r-j),\cdots, R_b^{(m)}(r-j)    \Big)\\ & \times  (\partial_1 \vec{P}_m)\Big(  R_b^{(2)}(r-k),\cdots, R_b^{(m)}(r-k)   \Big) R_b^{(2)\,'}(r-k)\,.
\end{align*}

\end{enumerate}

\end{theorem}

\begin{proof} Part (i).  By Corollary~\ref{CorVar}, as $n\rightarrow \infty$
$$\varrho_{n}^{(2)}\big( \beta_{n,r}^{(b)}\big)\quad\longrightarrow  \quad R_b^{(2)}(r)\,.   $$
Define $Y^{(n,r)}:=\varrho_{n}^{(2)}\big( \beta_{n,r}^{(b)}\big)$, and let the functions $F_{k}^{(x)}:\R^{m-1}\rightarrow \R^{m-1}$ be defined as in the proof of Lemma~\ref{LemFunGeneral}.
   Then I can write the vector of moments $\varrho_{n}^{(i)}\big( \beta_{n,r}^{(b)}\big)$ for $3\leq i\leq m$ in terms of $Y^{(n,r)}$ and the vector of moments $\varrho_{0}^{(i)}\big( \beta_{n,r}^{(b)}\big)$ for $3\leq i\leq m$ as
\begin{align}\label{Diffle}
\Big(\varrho_{n}^{(3)}\big( \beta_{n,r}^{(b)}\big),\cdots, \varrho_{n}^{(m)}\big( \beta_{n,r}^{(b)}\big)  \Big)\,=\, F_{1}^{(Y^{(n,r)})}\circ\cdots     \circ F_{n}^{(Y^{(n,r)})}\Big(\varrho_{0}^{(3)}\big( \beta_{n,r}^{(b)}\big),\cdots, \varrho_{0}^{(m)}\big( \beta_{n,r}^{(b)}\big)  \Big)\,.
\end{align}
For any fixed $\epsilon>0$, I can pick $-r>0$ large enough so that  $\varrho_{n}^{(2)}( \beta_{n,r}^{(b)})\leq\epsilon$  for large enough $n\in \mathbb{N}$.  Moreover, for  $n\gg 1$,
$$\varrho_{0}^{(i)}( \beta_{n,r}^{(b)})\,=\, \mathbb{E}\Bigg[\bigg( \frac{e^{ \beta_{n,r}^{(b)}\omega}}{\mathbb{E}\big[ e^{ \beta_{n,r}^{(b)}\omega}  \big]   }-1\bigg)^i\Bigg] \,=\, \mathit{O}\big(  n^{-\frac{i}{2}} \big) \,.    $$
Thus the above will be smaller in absolute value than any fixed $\epsilon>0$.  Pick $\epsilon>0$ so that the conclusions of part (i) of Lemma~\ref{LemFunGeneral} hold. By the triangle inequality,
\begin{align}\label{Yuigle}
\Big\| \Big(\varrho_{n}^{(3)}\big(& \beta_{n,r}^{(b)}\big), \cdots,\varrho_{n}^{(m)}\big( \beta_{n,r}^{(b)}\big)\Big) \,-\,\vec{H}_m\big(R_b^{(2)}(r)\big)\Big\|_\infty\nonumber \\   \,\leq \,&\Big\| \Big(\varrho_{n}^{(3)}\big( \beta_{n,r}^{(b)}\big),\cdots,\varrho_{n}^{(m)}\big( \beta_{n,r}^{(b)}\big)\Big) \,-\,F_{1}^{(R_b^{(2)}(r))}\circ\cdots     \circ F_{}^{(R_b^{(2)}(r))}\Big(\varrho_{0}^{(3)}\big( \beta_{n,r}^{(b)}\big),\cdots,\varrho_{0}^{(m)}\big( \beta_{n,r}^{(b)}\big)\Big) \Big\|_\infty \nonumber \\ & \,+\, \Big\|F_{1}^{(R_b^{(2)}(r))}\circ\cdots     \circ F_{n}^{(R_b^{(2)}(r))}\Big(\varrho_{0}^{(3)}\big( \beta_{n,r}^{(b)}\big),\cdots,\varrho_{0}^{(m)}\big( \beta_{n,r}^{(b)}\big)\Big)\,-\, \vec{H}_m\big(R_b^{(2)}(r)\big)  \Big\|_\infty \,.
\end{align}
 Then since  $\varrho_{n}^{(2)}\big( \beta_{n,r}^{(b)}\big)\leq \epsilon $ and $ \max_{1\leq i\leq m}|\varrho_{0}^{(i)}( \beta_{n,r}^{(b)})|\leq \epsilon $ hold for large enough $n$, the first term above is bounded by
\begin{align}
\Big\|  \Big(\varrho_{n}^{(3)}\big(& \beta_{n,r}^{(b)}\big),\cdots,\varrho_{n}^{(m)}\big( \beta_{n,r}^{(b)}\big)\Big) \,-\,F_{1}^{(R_b^{(2)}(r))}\circ\cdots     \circ F_{n}^{(R_b^{(2)}(r))}\Big(\varrho_{0}^{(3)}\big( \beta_{n,r}^{(b)}\big),\cdots,\varrho_{0}^{(m)}\big( \beta_{n,r}^{(b)}\big)\Big) \Big\|_\infty \nonumber\\
= \bigg\| &\Big( F_{1}^{(Y^{(n,r)})}\circ\cdots     \circ F_{n}^{(Y^{(n,r)})}(\mathbf{y})-F_{1}^{(R_b^{(2)}(r))}\circ\cdots     \circ F_{n}^{(R_b^{(2)}(r))}(\mathbf{y})\Big)\Big|_{\mathbf{y}=\big(\varrho_{0}^{(3)}( \beta_{n,r}^{(b)}),\cdots, \varrho_{0}^{(m)}( \beta_{n,r}^{(b)}) \big)}    \bigg\|_\infty \nonumber
 \\ &\,\leq\,\big| Y^{(n,r)}-R_b^{(2)}(r) \big|\sup_{ \substack{ x, \|\mathbf{y}\|_{\infty}\leq \epsilon \\ n\in \mathbb{N}  }}\Big\|\frac{d}{dx} F_{1}^{(x)}\circ\cdots     \circ F_{n}^{(x)}(\mathbf{y})\Big\|_{\infty}\,,\label{diz}
\end{align}
where the equality is from~(\ref{Diffle}).  The supremum on the bottom line is finite by part (v) of Lemma~\ref{LemFunGeneral}.  Thus~(\ref{diz}) goes to zero since $Y^{(n,r)}$ converges to $R_b^{(2)}(r)$.  The second term on the right side of~(\ref{Yuigle}) converges to zero by part (i) of Lemma~\ref{LemFunGeneral}.  Thus the vector $\big(\varrho_{n}^{(3)}( \beta_{n,r}^{(b)}),\cdots,\varrho_{n}^{(m)}( \beta_{n,r}^{(b)})\big)$ converges with large $n$ to $\vec{H}_m\big(R_b^{(2)}(r)\big)$.  \vspace{.3cm}

\noindent Part (ii).   By the above analysis,
$$R_b^{(m)}(r)\,=\,H_m\big(R_b^{(2)}(r)\big) \,. $$
 The result then follows from part (ii) of Lemma~\ref{LemFunGeneral}.

\vspace{.3cm}

\noindent  Part (iii). This result follows from part (iii) of Lemma~\ref{LemFunGeneral} since
 \begin{align*}
  P_m\big(R_b^{(2)}(r),\cdots , R_b^{(m)}(r)\big)\,=\,&  P_m\big(R_b^{(2)}(r), \vec{H}_m\big( R_b^{(2)}(r)\big)\big) \\ \,=\,& H_m\big(M_b\big(  R_b^{(2)}(r)\big)\big)\\  \,=\,& H_m\big(  R_b^{(2)}(r+1)\big)\,=\,   R_b^{(m)}(r+1)  \,. 
  \end{align*}

\vspace{.3cm}

\noindent Part  (iv).  Since  $P_m(\mathbf{y})=\frac{1}{b^{m-2}} y_m+ y_mU_m(\mathbf{y})+V_m(\mathbf{y})$, I can iterate part (ii) to get
\begin{align}
 R_b^{(m)}(r)\,=\,& \sum_{k=1}^{n-1}V_m\Big(R_b^{(2)}(r-k),\cdots, R_{b}^{(m-1)}(r-k)\Big)\nonumber \\  &  \times \prod_{j=1}^{k-1}\bigg(\frac{1}{b^{m-2}}\,+\, U_m\Big(R_b^{(2)}(r-j),\cdots, R_b^{(m)}(r-j)\Big)\bigg)\,
\nonumber \\ & \,+\, R_b^{(m)}(r-n) \prod_{j=1}^{n-1}\bigg(\frac{1}{b^{m-2}}\,+\, U_m\Big(R_b^{(2)}(r-j),\cdots, R_b^{(m)}(r-j)\Big)\bigg) \,.\label{Rexp}
\end{align}
Since $R_b^{(j)}(r)\searrow 0$ as $r\searrow -\infty $, and $U_m(\mathbf{y})\rightarrow 0$ as $\|\mathbf{y}\|_\infty\rightarrow  0$, the term on the bottom line of~(\ref{Rexp}) vanishes exponentially as $n\rightarrow \infty$ for large enough $-r>0$.

\vspace{.3cm}

\noindent Part (v).   Part (v) of Lemma~\ref{LemFunGeneral} yields
\begin{align*}
\big(R_b^{(3)\,'}(r),\cdots, R_b^{(m)\,'}(r)\big)\,=\,&\frac{d}{dr}\vec{H}_m\big(R_b^{(2)}(r)\big)\\ \,=\,&\sum_{k=1}^{\infty}  \prod_{j=1}^{k-1} (\mathbf{\widetilde{D}} \vec{P}_m)\Big(  M^{-j}_b(x)    ; \vec{H}_m\big( M^{-j}_b(x)  \big)    \Big)\\  &\times (\partial_1 \vec{P}_m)\Big(M^{-k}_b(x); \vec{H}_m\big(M^{-k}_b(x)\big)\Big)\frac{d}{dx}M^{-k}_b(x) \Big|_{x=R_b^{(2)}(r)}R^{(2)\,'}_b(r) \\
\,=\,&\sum_{k=1}^{\infty} \prod_{j=1}^{k-1} (\mathbf{\widetilde{D}} \vec{P}_m)\Big( R_b^{(2)}(r-j),\cdots, R_b^{(m)}(r-j)    \Big)\\  &\times  (\partial_1 \vec{P}_m)\Big( R_b^{(2)}(r-k),\cdots, R_b^{(m)}(r-k)   \Big) R^{(2)\,'}_b(r-k)\,.
\end{align*}

\end{proof}

\subsection{Asymptotics for the centered moments as $r\rightarrow -\infty$}\label{SecGaussian}

In this section, I will prove property (III) of Theorem~\ref{ThmMain}, which is a corollary of Theoerem~\ref{ThmAsy} below.  The following technical lemma is similar to Lemma~3.5 in~\cite{AC}, and the proof is placed in the appendix.


\begin{lemma}\label{LemMomBounds} Let $\beta_{n,r}^{(b)}$ be defined as in~(\ref{BetaForm}) and  fix $m\in \Z$ and $\lambda\in \R$.  There is a $C>0$ such that for all $r\in (-\infty,\lambda] $
\begin{align}\label{Rho}
\big|\varrho_ {k}^{(m)}\big(\beta_{n,r}^{(b)}\big)\big|\, \leq  \, C \Big(\varrho_ {k}^{(2)}\big(\beta_{n,r}^{(b)}\big)\Big)^{\frac{m}{2}}  \end{align}
for large enough $n\in \mathbb{N}$ and all $1\leq k\leq n $

\end{lemma}

\begin{theorem}[Gaussian behavior for $-r\gg 1$]\label{ThmAsy}
The function $R_b^{(m)}(r)$  has the following $r\rightarrow -\infty$ asymptotics:
\begin{enumerate}[(i).]
\item For any $m\in \mathbb{N}$ with $m\geq 2$,  $  R_b^{(m)}(r) =\mathit{O}\big( |r|^{-\lceil \frac{m}{2} \rceil} \big) $.

\item  For even $m\in \mathbb{N}$, $  R_b^{(m)}(r)=\kappa_b^{m}\frac{ m!  }{ 2^{\frac{m}{2}}(\frac{m}{2})! }|r|^{-\frac{m}{2}} + \mathit{O}\big(  |r|^{-\frac{m}{2}-1} \big)   $.

\end{enumerate}

\end{theorem}

\begin{proof} Part (i): By Theorem~\ref{ThmMomConv}, I know that for any $r\in \R$ and $m\in \mathbb{N}$ as $n\rightarrow \infty$
$$\varrho_{n}^{(m)}\big(\beta_{n,r}^{(b)}\big)\quad \longrightarrow \quad R^{(m)}_b(r) \,.  $$
 From Lemma~\ref{LemMomBounds}, for any $m\in \{2,3,4,\cdots\}$ there is a $C>0$ such that for any $r<-1$ and large enough $n$
\begin{align}
\big|\varrho_{n}^{(m)}\big(\beta_{n,r}^{(b)}\big)\big|\, \leq  \, C \Big(\varrho_ {n}^{(2)}\big(\beta_{n,r}^{(b)} \big)\Big)^{\frac{m}{2}} \,.
\end{align}
By Lemma~\ref{LemVar} I have $R^{(2)}_b(r) = -\frac{\kappa_b^2}{r}+\mathit{o}\big( |r|^{-\frac{3}{2}}  \big)$ for $-r\gg 1$, and thus $R^{(m)}_b(r)\,=\,\mathit{O}\big( |r|^{-\frac{m}{2} } \big)$ for $m\geq 2$.

For $m$ odd I will employ a strong induction argument. Suppose for the sake of induction that there is a $c>0$ such that for all $r<-1$ and $i\in \{1,\cdots, m-1\}$
\begin{align*}
R_b^{(i)}(r)\,\leq \, \frac{c}{  |r|^{\lceil i \rceil}  }  \,.
\end{align*}
 By part (iii) of Theorem~\ref{ThmMomConv}, I can represent $ R_b^{(m)}(r)$ in the series form
\begin{align}\label{PertSeries}
 R_b^{(m)}(r)\,=\, \sum_{k=1}^{\infty}&V_m\Big(R_b^{(2)}(r-k),\cdots, R_{b}^{(m-1)}(r-k)\Big)\nonumber \\  &  \times \prod_{j=1}^{k-1}\bigg(\frac{1}{b^{m-2}}\,+\, U_m\Big(R_b^{(2)}(r-j),\cdots, R_b^{(m)}(r-j)\Big)\bigg)\,,
\end{align}
where $U_m$ and $V_m$ are defined as in Proposition~\ref{PropPoly}.  By Proposition~\ref{PropPoly}, $y_mU_m(y_2,\cdots,y_m)$ and $V_m(y_2,\cdots,y_{m-1})$ are linear combinations of monomials $y_{i_1}\cdots y_{i_\ell}$ with $i_{1}+\cdots +i_\ell\geq m+1$ since $m$ is odd.  It follows that there is a $\mathbf{c}>0$ such that
\begin{align}\label{Vikochev}
V_m\Big(R_b^{(2)}(r),\cdots, R_{b}^{(m-1)}(r)\Big)\,\leq \, \mathbf{c}|r|^{-\frac{m+1}{2}}\hspace{.3cm}\text{and}\hspace{.3cm}U_m\Big(R_b^{(2)}(r),\cdots, R_{b}^{(m)}(r')\Big)\,\leq \, \mathbf{c}|r|^{-\frac{m-1}{2}}\,.
\end{align}
I can apply~(\ref{Vikochev}) to~(\ref{PertSeries}) to get the bound
\begin{align*}
 R_b^{(m)}(r) \,\leq \, \sum_{k=1}^{\infty} \frac{\mathbf{c}}{|-r+k|^{ \frac{m+1}{2}} }  \prod_{j=1}^{k-1}\bigg(\frac{1}{b^{m-2}}+ \frac{\mathbf{c}}{|-r+j|^{\frac{m-1}{2}}}   \bigg)\,.
\end{align*}
The above is $\mathit{O}\big(|r|^{-\frac{m+1}{2}}  \big)$, which completes the induction argument.

\vspace{.5cm}

\noindent Part (ii): I will again apply a strong induction argument by assuming that (ii) holds 
for all even $j$ less than $m$. By part (iii) of Proposition~\ref{PropPoly},
\begin{align}
V_m\big(R_b^{(2)}(r),&\cdots, R_b^{(m-1)}(r)\big)\nonumber \\ \,=\,& \frac{1}{b^{m}}\frac{d^m}{dx^m}\Bigg( 1+\sum_{1\leq j < \frac{m}{2}} R_b^{(2j)}(r)\frac{x^{2j}}{(2j)!}  \Bigg)^{b^2}\Bigg|_{x=0}\,+\, \mathit{O}\Big( |r|^{- \frac{m}{2}-1} \Big)\nonumber \\
\,=\,& \frac{1}{b^{m}}\frac{d^m}{dx^m}\Bigg( 1+\sum_{1\leq j < \frac{m}{2}} \kappa_b^{2j}\frac{ (2j)!  }{ 2^j j!|r|^j }\frac{x^{2j}}{(2j)!}  \Bigg)^{b^2}\Bigg|_{x=0}\,+\, \mathit{O}\Big( |r|^{- \frac{m}{2}-1} \Big)\,.\nonumber 
\intertext{Completing the series with  the $j\geq \frac{m}{2}$ terms, yields  a nonzero contribution for only the $j=\frac{m}{2}$ term, so I can write    }
\,=\,& \frac{1}{b^{m}}\frac{d^m}{dx^m}\Bigg( 1+\sum_{j=1}^{\infty} \frac{ \kappa_b^{2j} }{ 2^j j! |r|^j}x^{2j}  \Bigg)^{b^2}\Bigg|_{x=0}\,-\,\frac{1}{b^{m-2}} \kappa_b^{2j}\frac{ m! }{ 2^{\frac{m}{2}} (\frac{m}{2})! |r|^{\frac{m}{2}}}\,+\, \mathit{O}\Big( |r|^{- \frac{m}{2}-1} \Big) \nonumber \\ 
\,=\,& \frac{1}{b^{m}}\frac{d^m}{dx^m} e^{b^2 \frac{\kappa_b^2x^2}{ 2|r|  }   } \Bigg|_{x=0}\,-\,\frac{1}{b^{m-2}} \frac{\kappa_b^{m} m! }{ 2^{\frac{m}{2}} (\frac{m}{2})! }|r|^{-\frac{m}{2}}\,+\, \mathit{O}\Big( |r|^{- \frac{m}{2}-1} \Big)
\nonumber \\
\,=\,&\bigg( 1\,-\,\frac{1}{b^{m-2}}   \bigg)\frac{\kappa_b^{m} m!}{ 2^{\frac{m}{2}} (\frac{m}{2})! |r|^{\frac{m}{2}} }  \,+\, \mathit{O}\Big( |r|^{- \frac{m}{2}-1} \Big)\,.\label{Haffle}
\end{align}
By~(\ref{PertSeries}) and~(\ref{Haffle}), I have the first equality below for $-r\gg 1$
\begin{align*}
 R_b^{(m)}(r)\,=\,& \sum_{k=1}^{\infty}\Bigg( \bigg( 1\,-\,\frac{1}{b^{m-2}}   \bigg)\frac{\kappa_b^{m} m!}{ 2^{\frac{m}{2}} (\frac{m}{2})! |-r+k|^{\frac{m}{2}} }  \,+\, \mathit{O}\Big( |-r+k|^{- \frac{m}{2}-1} \Big)\Bigg)\nonumber \\  &  \times \prod_{j=1}^{k-1}\bigg(\frac{1}{b^{m-2}}\,+\,  \mathit{O}\Big( |-r+j|^{- \frac{m}{2}-1} \Big)\bigg)\,\\
 \,=\,&  \bigg( 1\,-\,\frac{1}{b^{m-2}}   \bigg)\frac{\kappa_b^{m} m!}{ 2^{\frac{m}{2}} (\frac{m}{2})! |r|^{\frac{m}{2}} }\sum_{k=1}^{\infty}\Big(\frac{1}{b^{m-2}}\Big)^{k-1}   \,+\, \mathit{O}\Big( |r|^{- \frac{m}{2}-1} \Big)\\
  \,=\,& \frac{\kappa_b^{m} m!}{ 2^{\frac{m}{2}} (\frac{m}{2})! |r|^{\frac{m}{2}} }  \,+\, \mathit{O}\Big( |r|^{- \frac{m}{2}-1} \Big)\,,
\end{align*}
which completes the proof.

\end{proof}

\begin{appendix}

\section{Proof of Lemma~\ref{LemMomBounds}  }

\begin{proof}[Proof of Lemma~\ref{LemMomBounds}] I will apply a strong induction argument in $\mathbf{m}=2, 3, 4, \cdots$.    Note that the base case $m=2$ holds trivially. Now  let us assume for the  purpose of  induction that the statement of the lemma holds for all $m\in \{2,\dots, \mathbf{m}-1\}$, in other terms, there is a $c\equiv c(r,\mathbf{m})>0$ such that for any $r\in (-\infty,\lambda] $ and $m\in\{2,\cdots,\mathbf{m}\}$ the inequality
\begin{align}\label{Rho}
\big|\varrho_ {k}^{(m)}\big(\beta_{n,r}^{(b)}\big)\big|\, \leq  \, c \Big(\varrho_ {k}^{(2)}\big(\beta_{n,r}^{(b)}\big)\Big)^{\frac{m}{2}} 
 \end{align}
 holds for large enough $n\in \mathbb{N}$ and all $1\leq k\leq n $.

 By part (i) of Proposition~\ref{PropPoly}, I have the recursive inequality
\begin{align}
 \big|\varrho_ {k+1}^{(\mathbf{m})}\big(\beta_{n,r}^{(b)}\big) \big|\, < \,&\frac{1}{b^{\mathbf{m}-2}}\big|\varrho_ {k}^{(\mathbf{m})}\big(\beta_{n,r}^{(b)}\big)\big|\,+\,
 \big|\varrho_ {k}^{(\mathbf{m})}\big(\beta_{n,r}^{(b)}\big) \big|\,\Big|U_\mathbf{m}\Big(  \varrho_ {k}^{(2)}\big(\beta_{n,r}^{(b)}\big), \cdots, \varrho_ {k}^{(\mathbf{m})}\big(\beta_{n,r}^{(b)}\big)  \Big) \Big| \nonumber \\  &\,+\,\Big|V_\mathbf{m}\Big( \varrho_ {k}^{(2)}\big(\beta_{n,r}^{(b)}\big), \cdots,\varrho_ {k}^{(\mathbf{m}-1)}\big(\beta_{n,r}^{(b)}\big)  \Big)\Big|  \,.\label{Hugh}
\end{align}
  Since the variable $\omega$ has finite exponential moments,   I have that  
  \begin{align}\label{Yambell}
  \big|\varrho_ {0}^{(\mathbf{m})}\big(\beta_{n,r}^{(b)}\big)\big| \,=\,\Bigg|\mathbb{E}\Bigg[\bigg(\frac{ \exp\big\{\omega \beta_{n,r}^{(b)} \big\} }{\mathbb{E}\big[ \exp\big\{\omega \beta_{n,r}^{(b)} \big\}   \big]     }-1   \bigg)^{\mathbf{m}}\Bigg]\Bigg| =\mathit{O}\Big(\frac{1}{n^{\frac{\mathbf{m}}{2}}}   \Big)
  \end{align}
   is small with large $n$.   Since $\varrho_ {n}^{(2)}\big(\beta_{n,r}^{(b)}\big)$ converges with large $n$ to $R_b(r)$, and $R_b(r)=-\frac{\kappa_b^2}{r}+\mathit{o}(r^{-1})$ for $-r\gg 1$  by part (iii) of Lemma~\ref{LemVar}, there is a $-\widetilde{\lambda}_1>2\kappa_b^2$ large enough so that when $r\in (-\infty,\widetilde{\lambda}_1]$, then 
  \begin{align}\label{Gorg}
 \varrho_ {n}^{(2)}\big(\beta_{n,r}^{(b)}\big)\,<\,-\frac{2\kappa_b^2}{r}\,<\,1\,
\end{align} 
   holds for large enough $n$.

 By the induction assumption, the term $\big|V_\mathbf{m}\big(  \varrho_ {k}^{(2)}\big(\beta_{n,r}^{(b)}\big), \cdots, \varrho_ {k}^{(\mathbf{m}-1)}\big(\beta_{n,r}^{(b)}\big)   \big)\big| $  has the bound
   \begin{align}\label{zum}
  \Big|V_\mathbf{m}\Big(  \varrho_ {k}^{(2)}\big(\beta_{n,r}^{(b)}\big), \cdots, \varrho_ {k}^{(\mathbf{m}-1)}\big(\beta_{n,r}^{(b)}\big)   \Big)\Big| \,\leq &\,V_\mathbf{m}\bigg(   \varrho_{k}^{(2)}\big(\beta_{n,r}^{(b)}\big), \cdots,  \Big(\varrho_{k}^{(2)}\big(\beta_{n,r}^{(b)}\big)\Big)^{\frac{\mathbf{m-1}}{2}} \bigg) \nonumber \\  \,\leq  &\, \widehat{c}\big[\varrho_{k}^{(2)}\big(\beta_{n,r}^{(b)}\big)\big]^{\frac{\mathbf{m}}{2}} \,.
  \end{align}
 If $r\in (-\infty,\widetilde{\lambda}_1]$, then $\varrho_{k}^{(2)}\big(\beta_{n,r}^{(b)}\big)<1$ for large $n$ and the above  is bounded by a constant multiple  $\widehat{c} $ of $\big[\varrho_{k}^{(2)}\big(\beta_{n,r}^{(b)}\big)\big]^{\mathbf{m}/2}$ as a consequence of part (ii) of Proposition~\ref{PropPoly}.  
  
    If $ \big|\varrho_ {k}^{(\mathbf{m})}\big(\beta_{n,r}^{(b)}\big) \big|   \leq  \mathbf{x} $ for some $\mathbf{x}<1$,  the factor $\big|U_\mathbf{m}\big(  \varrho_ {k}^{(2)}\big(\beta_{n,r}^{(b)}\big), \cdots, \varrho_ {k}^{(\mathbf{m})}\big(\beta_{n,r}^{(b)}\big)  \big)\big|$ in~(\ref{Hugh}) has a bound of the form
\begin{align}\label{UINEQ}
\Big|U_\mathbf{m}\Big(  \varrho_ {k}^{(2)}\big(\beta_{n,r}^{(b)}\big), \cdots, \varrho_ {k}^{(\mathbf{m})}\big(\beta_{n,r}^{(b)}\big) \Big)\Big|\,\leq \, U_\mathbf{m}\bigg(   \frac{2\kappa_b^2}{|r|}, \cdots,  \Big(\frac{2\kappa_b^{2}}{|r|}\Big)^{\frac{\mathbf{m}-1}{2}}, \mathbf{x}  \bigg)  \,\leq \,\mathbf{c}\mathbf{x} \,+\, \frac{\mathbf{c}}{|r|}\,,
\end{align}
where the first inequality holds by~(\ref{Gorg}) for sufficiently large $n\in \mathbb{N}$ and all $1\leq k \leq n$ when $r\leq \widetilde{\lambda}_2$.  The second inequality holds for some $\mathbf{c}>0$ and all $\mathbf{x}<1$ and $r\in (-\infty, \widetilde{\lambda}_2]$ since the polynomial $U_{\mathbf{m}}$ has no constant term
 by (i) of Proposition~\ref{PropPoly}.  Pick some $\mathbf{x}$ small enough so  that 
\begin{align*}
\delta\,:=\, \mathbf{c}\mathbf{x}+\frac{1}{b^{\mathbf{m}-2}} \,<\,1 \,. 
\end{align*}
Let $\widehat{k}_n\in \mathbb{N}$ be the smallest value  such that $\big| \varrho_ {\widehat{k}_n}^{(\mathbf{m})}\big(\beta_{n,r}^{(b)}\big)\big|> \mathbf{x} $. 

  Define $\lambda= \min( \widetilde{\lambda}_1, \widetilde{\lambda}_2 )$. By the  bounds~(\ref{zum}) and~(\ref{UINEQ}), there is a $\delta\in (0,1)$ and a $C'>0$ such that for all $r\in (-\infty, \lambda)$ there is an  $n\in \mathbb{N}$ large enough so that for all $k<\min(n,\widehat{k}_n)$ 
\begin{align}\label{BowWow}
 \big| \varrho_ {k+1}^{(\mathbf{m})}\big(\beta_{n,r}^{(b)}\big)\big| \, \leq  \, \delta  \big|\varrho_ {k}^{(\mathbf{m})}\big(\beta_{n,r}^{(b)}\big)\big|\,+\,C'\Big(\varrho_{k}^{(2)}\big(\beta_{n,r}^{(b)}\big)\Big)^{\frac{\mathbf{m}}{2}}\,.
\end{align}
   Using~(\ref{BowWow}) recursively,  it follows that for $k\leq \min(n,\widehat{k}_n)$
\begin{align}
 \big|\varrho_ {k}^{(\mathbf{m})}\big(\beta_{n,r}^{(b)}\big)\big| \, \leq  & \, \delta^{k}\big|\varrho_ {0}^{(\mathbf{m})}\big(\beta_{n,r}^{(b)}\big)\big|\,+\,C' \sum_{j=0}^{k-1}\delta^{k-1-j}\Big(\varrho_{j}^{(2)}\big(\beta_{n,r}^{(b)}\big)\Big)^{\frac{\mathbf{m}}{2}}\nonumber \\  \,\leq & \, \mathit{O}\Big(\frac{1}{n^{\frac{\mathbf{m}}{2}}}   \Big)  \,+\,\frac{C'}{1-\delta}\left(\varrho_{k}^{(2)}\big(\beta_{n,r}^{(b)}\big)\right)^{\frac{\mathbf{m}}{2}}\nonumber \\ \,\leq &\, C\Big(\varrho_{k}^{(2)}\big(\beta_{n,r}^{(b)}\big)\Big)^{\frac{\mathbf{m}}{2}}\,.\label{HERE}
\end{align}
The second inequality holds by~(\ref{Yambell}) for the first term and  since $\varrho_{k}^{(2)}\big(\beta_{n,r}^{(b)}\big) $ is increasing with $k\in \mathbb{N}$ for the second term.  The third inequality holds for some $C>0$ since a multiple of $\big(\varrho_{k}^{(2)}\big(\beta_{n,r}^{(b)}\big)\big)^{\frac{\mathbf{m}}{2}}$ can be used to cover the error $\mathit{O}\big(n^{-\frac{\mathbf{m}}{2}}  \big) $.  To understand this, recall that $\varrho_{0}^{(2)}\big(\beta_{n,r}^{(b)}\big)=\kappa_b^2/n+\mathit{o}(1/n)$ for $n\gg1$ and, again, invoke that $\varrho_{k}^{(2)}\big(\beta_{n,r}^{(b)}\big) $ is increasing with $k$.  If $-\lambda>0 $ is large enough so that 
\begin{align}\label{DELTA}
\delta\,:=\, \mathbf{c}C\big(R_b(\lambda)\big)^{\frac{\mathbf{m}}{2}}+\frac{1}{b^{\mathbf{m}-2}} \,<\,1 \,, 
\end{align}
then for all $r\in (-\infty,\lambda)$ the inequality $\widehat{k}_n>n$ will hold for sufficiently large $n$ since $\varrho_{k}^{(2)}\big(\beta_{n,r}^{(b)}\big)$ converges to   $ R_b(\lambda)$ as $n\rightarrow \infty$.  Hence,~(\ref{HERE}) will hold for all $k\leq n$ when $n$ is large enough.

The above argument  proves the statement of the lemma is true for $\lambda\in \R$ sufficiently far in the negative direction so that~(\ref{DELTA}) holds.   For arbitrary $\lambda'\in \R$, pick $N\in \mathbb{N}$ large enough so that $\lambda =\lambda'-N $ satisfies~(\ref{DELTA}).  Then the reasoning above applies in the same way to show that $\widehat{k}_n> n-N$   for large $n$, and thus~(\ref{HERE}) will hold for all $1\leq k\leq n-N$.   The inequality can then be extended to $n-N< k\leq n$     through~(\ref{Hugh}).

\end{proof}

\end{appendix}


\begin{thebibliography}{99}

\bibitem{AC} T.\ Alberts, J.\ Clark: \emph{Nested critical points for a directed polymer on a disordered diamond lattice}  (to appear in Journal of Theoretical Probability)  \textup{arXiv:}1602.06629 (2016).


\bibitem{ACK} T.\ Alberts, J.\ Clark, S.\ Koci\'c: \emph{The intermediate disorder regime for a directed polymer model on a hierarchical lattice} (to appear in Stochastic Processes and their Applications) \textup{arXiv:}1508.04791 (2015).

\bibitem{AKQ} T.\ Alberts, K.\ Khanin, J.\ Quastel: \emph{The intermediate disorder regime for directed polymers in dimension $1+1$}, Ann.\ Probab.\ \textbf{42}, No. 3, 1212-1256 (2014).

\bibitem{AKQII} T.\ Alberts, K.\ Khanin, J.\ Quastel: \emph{The continuum directed random polymer}, J. Stat. Phys. \textbf{154}, No.\ 1-2, 305-326 (2014).


\bibitem{Bolth} E.\ Bolthausen: \textit{A note on the diffusion of directed polymers in a random environment}, Comm. Math. Phys. \textbf{123} 529-534 (1989).



\bibitem{CSZ1} F.\ Caravenna, R.\ Sun, and N.\ Zygouras: \textit{Polynomial chaos and scaling limits of disordered systems},  J. Eur. Math. Soc. \textbf{19}, 1-65 (2017).


\bibitem{CSZ3} F.\ Caravenna, R.\ Sun, and N.\ Zygouras: \textit{Universality in marginally relevant disordered systems}, arXiv:1510.06287 [math.PR] (2015).






\bibitem{CometsBook} F.\ Comets: \textit{Directed Polymers in Random Environments}, Lecture Notes in Mathematics, Vol.\  \textbf{2175}, Springer 2017.

\bibitem{CometsII} F.\ Comets, N.\ Yoshida: \textit{Directed polymers in random environment are diffusive at weak disorder}, Ann.\ Probab.\, \textbf{34}, 1746-1770 (2006).



\bibitem{Cook} J.\ Cook, B.\ Derrida: \emph{Polymers on disordered hierarchical lattices:  a nonlinear combination of random variables}, J. Stat. Phys. \textbf{57} 89-139 (1989).


\bibitem{Gardner} B.\ Derrida, E.\ Gardner: \emph{Renormalisation group study of a disordered model}, J. Phys. A: Math. Gen. \textbf{17}, 3223-3236 (1984).


\bibitem{Derrida} B.\ Derrida, R.B.\ Griffith: \emph{Directed polymers on disordered hierarchical lattices}, Europhys. Lett. \textbf{8}, No. 2, 111-116 (1989).

\bibitem{Hakim} B.\ Derrida, V.\ Hakim, J.\ Vannimenius: \emph{Effect of disorder on two-dimensional wetting}, J. Stat. Phys. \textbf{66} 1189-1213 (1992).




\bibitem{GLT}  G.\ Giacomin, H.\ Lacoin, F.L.\ Toninelli: \emph{Hierarchical pinning models, quadratic maps, and quenched disorder}, Probab. Theor. Rel. Fields \textbf{145}, (2009).

\bibitem{Goldstein} L.\ Goldstein: \emph{Normal approximation for hierarchical structures},  Ann.\ Appl.\ Probab.\ \textbf{14}, no.\ 4, 1950-1969 (2004).

\bibitem{Griffiths} R.B.\ Griffith, M.\ Kaufman: \emph{Spin systems on hierarchical lattices.  Introduction and thermodynamical limit}, Phys. Rev. B, \textbf{3} 26, no. 9, 5022-5032 (1982).

\bibitem{Hambly} B.M.\ Hambly, J.H.\ Jordan: \emph{A random hierarchical lattice: the series-parallel graph and its properties}, Adv.\ Appl.\ Prob., \textbf{36}, 824-838 (2004).

\bibitem{HamblyII}  B.M.\ Hambly, T.\ Kumagai: \emph{Diffusion on the scaling limit of the critical percolation cluster in the diamond hierarchical lattice}, Adv. Appl. Prob., \textbf{36}, 824-838 (2004).



\bibitem{lacoin} H.\ Lacoin, G.\ Moreno:  \emph{Directed Polymers on hierarchical lattices with site disorder}, Stoch. Proc. Appl. \textbf{120}, No. 4, 467-493 (2010).

\bibitem{lacoin3} H.\ Lacoin: \emph{Hierarchical pinning model with site disorder: disorder is marginally relevant}, Probab.\ Theor.\ Rel.\ Fields \textbf{148}, No.\ 1-2, 159-175 (2010).





\bibitem{Spohn} T.\ Schl\"osser, H.\ Spohn: \emph{Sample to sample fluctuations in the conductivity of a disordered medium}, J. Statist. Phys. \textbf{69}, 955-967 (1992).




\bibitem{Wehr}  J.\ Wehr, J.M.\ Woo: \emph{Central limit theorems for nonlinear hierarchical sequences or random variables}, J.\ Statist.\ Phys.\ \textbf{104}, 777-797 (2001).

\end{thebibliography}
\end{document}